\documentclass[reqno]{amsart}
\usepackage{amssymb}
\usepackage{hyperref}
\usepackage{booktabs}
\usepackage{multirow}
\usepackage{tikz}

\newtheorem{theorem}{Theorem}[section]
\newtheorem{proposition}[theorem]{Proposition}
\newtheorem{lemma}[theorem]{Lemma}
\newtheorem{corollary}[theorem]{Corollary}

\theoremstyle{remark}

\numberwithin{equation}{section}

\begin{document}

\title[Elliptic Ruijsenaars operators and Wess-Zumino-Witten fusion rings]
{Elliptic Ruijsenaars difference operators, symmetric polynomials, and Wess-Zumino-Witten fusion rings}

\author{Jan Felipe  van Diejen}

\address{
Instituto de Matem\'aticas, Universidad de Talca,
Casilla 747, Talca, Chile}

\email{diejen@inst-mat.utalca.cl}

\author{Tam\'as G\"orbe}

\address{School of Mathematics, University of Leeds, Leeds LS2 9JT,
UK}

\email{T.Gorbe@leeds.ac.uk}



\date{June 2021}

\begin{abstract} 
The fusion ring for $\widehat{\mathfrak{su}}(n)_m$ Wess-Zumino-Witten conformal field theories is known to be isomorphic to a factor ring
of the ring of symmetric polynomials presented by Schur polynomials. We introduce a deformation of this factor ring 
associated with eigenpolynomials for the elliptic Ruijsenaars difference operators. 
The corresponding Littlewood-Richardson coefficients are governed by a Pieri rule stemming from the eigenvalue equation.
The orthogonality of the eigenbasis gives rise to an analog of the Verlinde formula.
In the trigonometric limit, our construction recovers the refined $\widehat{\mathfrak{su}}(n)_m$ Wess-Zumino-Witten fusion ring associated with the Macdonald polynomials.
\end{abstract}

\maketitle

\section{Introduction}\label{sec1}
It is well-known that the structure constants for the ring of symmetric polynomials in $n$ variables in the basis of Schur polynomials $s_\lambda(x_1,\ldots ,x_n)$, the Littlewood-Richardson coefficients, count the tensor multiplicities in the decomposition of tensor products of irreducible representations for the Lie algebra $\mathfrak{su}(n;\mathbb{C})$
(cf. e.g. \cite{mac:symmetric,pro:lie}). 
Here the partitions $\lambda=(\lambda_1,\ldots,\lambda_n)$ label the dominant weights and one divides out by an ideal generated by the relation $x_1x_2\cdots x_n=1$.
If additionally---upon fixing a level $m\in\mathbb{N}$---the ideal
generated by the Schur polynomials corresponding to partitions such that $\lambda_1-\lambda_n=m+1$ is divided out, then one arrives at a finite-dimensional factor ring that is isomorphic to
the fusion ring for $\widehat{\mathfrak{su}}(n)_m$ Wess-Zumino-Witten conformal field theories  
\cite{dif-mat-sen:conformal,fuc:affine,gep:fusion,goo-wen:littlewood,kac:infinite,kor-str:slnk}.  The factor ring in question is also often referred to as the Verlinde algebra in the literature
and it has a natural basis of Schur classes labeled by partitions $\lambda$ with $\lambda_1-\lambda_n\leq m$ (which encode the dominant weights of the underlying affine Lie algebra $\widehat{\mathfrak{su}}(n)_m$).
The corresponding structure constants, i.e. the affine counterparts of the Littlewood-Richard coefficients at level $m$, describe the fusion rules for  primary fields of the associated conformal field theories.  From a mathematical point of view these structure constants compute, cf.  e.g. \cite{goo-nak:fusion},  the Littlewood-Richardson coefficients
of Hecke algebras at roots of unity \cite{goo-wen:littlewood} and the dimensions of spaces of conformal blocks of three-point functions
in Wess-Zumino-Witten conformal field theories \cite{tsu-ue-yam:conformal}.

Generalizations of the  $\widehat{\mathfrak{su}}(n)_m$ fusion ring have been constructed by means of parameter deformations of the Schur polynomials; important examples arise  this way from the Hall-Littlewood polynomials  \cite{kor:cylindric} and from the Macdonald polynomials  \cite{aga-sha:knot,che:double,kir:inner,nak:refined}. While the deformed fusion ring stemming from the Hall-Littlewood polynomials has been related to a deformation of the Verlinde algebra appearing in connection with an index formula of Teleman and Woodward
 \cite{and-guk-pei:verlinde,oku-yos:gauged,tel:k-theory,tel-woo:index}, in the case of the Macdonald polynomials one deals rather with a Verlinde algebra that is intimately connected to the computation of refined Chern-Simons  invariants for torus knots \cite{aga-sha:knot,che:daha,hag:combinatorics,gor-neg:refined,nak:refined}. 
 
From the point of view of quantum integrable particle dynamics, the Verlinde algebra for $\widehat{\mathfrak{su}}(n)_m$ Wess-Zumino-Witten conformal field theories
can be interpreted as a Hilbert space for  the phase model of impenetrable bosons on the one-dimensional periodic lattice  \cite{kor-str:slnk}. From this perspective, the deformations of the Verlinde algebra associated with the Hall-Littlewood polynomials and with the Macdonald polynomials correspond in turn to Hilbert spaces for the periodic $q$-boson model \cite{kor:cylindric} and for the quantized
trigonometric Ruijsenaars-Schneider system on $\mathbb{C}\mathbb{P}^{n-1}$ \cite{die-vin:quantum,gor-hal:quantization} (cf. also \cite{blo-des-mat:supersymmetric}), respectively.

The idea of the present work is to introduce a further generalization of the Verlinde algebra originating from the \emph{elliptic} Ruijsenaars operators; these are commuting difference operators 
with coefficients built from products of Weierstrass' sigma functions
\cite{rui:complete,rui:systems} that reduce to the Macdonald difference operators
\cite{mac:symmetric} in the trigonometric limit. The study of their eigenfunctions points towards an
elliptic counterpart of Macdonald's theory of symmetric polynomials \cite{eti-kir:affine,lan-nou-shi:construction,mir-mor-zen:duality,rai-sun-var:affine}. 
Following \cite{die-gor:elliptic}, we will discretize the Ruijsenaars operators on a lattice of points labeled by partitions.  Specifically, the lattice points are given by the $\mathfrak{su}(n)$ dominant weights shifted by a Weyl vector that is rescaled linearly (as usual) with the multiplicity (or coupling) parameter.
The eigenfunctions of these discrete Ruijsenaars operators  are subsequently constructed in terms of polynomials determined by a recurrence stemming from the eigenvalue equation.
The polynomials at issue turn out to provide a basis for the ring of symmetric polynomials. The associated Littlewood-Richardson coefficients constitute an elliptic deformation of Macdonald's $(q,t)$-Littlewood-Richardson coefficients \cite{mac:symmetric}. In this context, the eigenvalue equation for the elliptic Ruijsenaars operator entails  an explicit formula for the elliptic deformation of the Pieri rule. 
Next, we construct the corresponding elliptic analog of the fusion ring and compute its structure constants in terms of elliptic Littlewood-Richardson coefficients.
Now this deformed Verlinde algebra constitutes the Hilbert space for the
compact elliptic  Ruijsenaars model on $\mathbb{C}\mathbb{P}^{n-1}$, whose classical and quantum dynamics was studied in \cite{feh-gor:trigonometric} and \cite{die-gor:elliptic}, respectively.
This particle interpretation gives rise to an orthogonality relation for the eigenpolynomials of the elliptic Ruijsenaars operators \cite{die-gor:elliptic}, from which we derive a Verlinde formula for the structure constants upon identifying the pertinent elliptic deformation of the $\widehat{\mathfrak{su}}(n)_m$ Kac-Peterson modular $S$-matrix \cite{kac:infinite}. 

The material is organized as follows.
We start by discretizing the elliptic Ruijsenaars operators onto partitions in Section \ref{sec2}.
Next, in Section \ref{sec3},
we construct an eigenbasis for these discrete Ruijsenaars operators; this gives rise to an elliptic deformation of Littlewood-Richardson coefficients generated by
explicit Pieri rules.
By dividing out the ideal generated by basis polynomials labeled by partitions with $\lambda_1-\lambda_n=m+1$, we arrive  in Section \ref{sec4} at  an elliptic deformation of the
 $\widehat{\mathfrak{su}}(n)_m$ Wess-Zumino-Witten  fusion ring; we compute its structure constants in terms of elliptic Littlewood-Richardson coefficients.
In Section \ref{sec5},  a Verlinde formula for these structure constants is presented in terms of
the relevant elliptic deformation of the Kac-Peterson modular $S$-matrix. We wrap up by detailing how the structure constants for the 
 $\widehat{\mathfrak{su}}(n)_m$ Wess-Zumino-Witten  fusion ring and its refined deformation associated with the Macdonald polynomials are recovered from their elliptic counterparts via parameter degenerations.

\section{Discrete Ruijsenaars operators on partitions}\label{sec2}

\subsection{Ruijsenaars' commuting difference operators}
The Ruijsenaars difference operators $D_1,\ldots ,D_n$ are of the form
\begin{equation}\label{Dr}
D_{r}=\sum_{\substack{J\subset\{1,\dots,n\}\\|J|=r}} V_J(x)
T_J , \quad V_J(x)= \prod_{\substack{j\in J\\k\notin J}}\frac{[x_j-x_k+ \mathrm{g} ]}{[x_j-x_k]}   \qquad (r=1,\ldots ,n).
\end{equation}
Here $|\cdot |$ denotes the cardinality of the set in question and $T_J$ acts by translation on  complex functions $f(x)=f(x_1,\ldots,x_{n})$:
\begin{equation*}
(T_Jf)(x)=f(x+ \varepsilon_J)\quad\text{with}\ \varepsilon_J=\sum_{j\in J} \varepsilon_j 
\end{equation*}
(where $\varepsilon_1,\ldots ,\varepsilon_{n}$ refer to the standard unit basis of $\mathbb{C}^{n}$). 
For $\mathrm{g}\in\mathbb{C}$, these difference operators commute if the coefficients are built from a function
$[z]$ that factorizes into the product of a Weierstrass sigma function 
$\sigma (z)$ and a Gaussian of the form $\exp (az^2+bz)$ (for any $a,b\in\mathbb{C}$)  \cite{rui:complete,rui:systems}.
For our purposes it is convenient to pick  
 \begin{subequations}
 \begin{equation}\label{scaled-theta}
[z]=[z;p]=\dfrac{\vartheta_1(\frac{\alpha}{2}z ;p)}{\frac{\alpha}{2}\vartheta'_1(0;p)}\qquad (z\in \mathbb{C},\, \alpha >0, \ 0<p<1),
\end{equation}
where $\vartheta_1$  denotes the Jacobi theta function
\begin{align}\label{theta}
 \vartheta_1(z;p)&=2\sum_{l\geq 0} (-1)^l p^{(l+\frac{1}{2})^2}\sin(2l+1)z,\\
&=2p^{1/4}\sin(z)\prod_{l\geq 1} (1-p^{2l})(1-2p^{2l}\cos(2z)+p^{4l} ). \nonumber
\end{align}
\end{subequations}
The conversion to the Weierstrass sigma function associated with the period  lattice $\Omega=2\omega_1 \mathbb{Z}+2\omega_2 \mathbb{Z}$ is governed by the relation
(cf. e.g. \cite[\S 23.6(i)]{olv-loz-boi-cla:nist}):
\begin{equation*}
[z;p]= \sigma (z) e^{-\tfrac{\zeta (\omega_1)}{2\omega_1}z^2} 
\end{equation*}
with $\alpha=\frac{\pi}{\omega_1}$, $p=e^{i\pi \tau}$, $\tau= \frac{\omega_2}{\omega_1} $, and $\zeta(z)=\frac{\sigma^\prime(z)}{\sigma(z)}$.
 Below we will often use that $[z;p]$ extends analytically in $p$ to the interval $-1<p<1$ with $[z;0]=\frac{2}{\alpha}\sin (\frac{\alpha z}{2})$.

\subsection{Discretization on partitions}
Let $\lambda=(\lambda_1,\ldots,\lambda_n)$ denote a partition of length $\ell(\lambda)\leq n$, i.e. $\lambda$ belongs to
\begin{equation}
\Lambda^{(n)}=\{  \lambda\in\mathbb{Z}^n \mid \lambda_1\geq\lambda_2\geq\cdots\geq\lambda_n\geq 0\} .
\end{equation}
Unless explicitly stated otherwise, it will be assumed that the value of  $\mathrm{g}$ is chosen generically in $\mathbb{R}$ such that 
\begin{subequations}
\begin{equation}\label{g-regular}
\boxed{j\mathrm{g}\not\in \mathbb{Z}_{\leq 0}+{\textstyle \frac{2\pi}{\alpha}}\mathbb{Z}\quad\text{for}\  j=1,\ldots ,n}
\end{equation}
(where $\mathbb{Z}_{\leq 0}=\mathbb{Z}\setminus\mathbb{N}=\{0,-1,-2,-3,\ldots \}$), which
ensures in particular that
\begin{equation}\label{denom-regular}
 \prod_{1\leq j<k\leq n}{\textstyle  [ \lambda_j-\lambda_k+(k-j)\mathrm{g} ] } \neq 0\qquad (\forall\lambda\in\Lambda^{(n)}).
\end{equation}
\end{subequations}
The following lemma now allows us to restrict $D_r$ \eqref{Dr} to a discrete difference operator acting on functions supported on partitions shifted by
\begin{equation}
\rho_\mathrm{g}=\bigl( (n-1)\mathrm{g},(n-2)\mathrm{g},\ldots ,\mathrm{g},0\bigr) .
\end{equation}

\begin{lemma}[Boundary Condition]\label{discretization:lem}
For any $\lambda\in\Lambda^{(n)}$ and $J\subset \{ 1,\ldots ,n\}$ (with $\mathrm{g}$ generic as detailed above), one has that
\begin{equation}
V_J(\rho_\mathrm{g} +\lambda)=0\quad\text{if}\ \lambda+\varepsilon_J\not\in \Lambda^{(n)} .
\end{equation}
\end{lemma}
\begin{proof}
If $\lambda \in\Lambda^{(n)}$ and $\mu=\lambda+\varepsilon_J\not\in \Lambda^{(n)}$
then $\mu_j-\mu_{j+1}<0$ for some $1\leq j<n$, which implies that  $j\not\in J$,  $j+1\in J$ and $\lambda_j-\lambda_{j+1}=0$.
Since the denominators do not vanish because of Eq. \eqref{denom-regular}, one then picks up a zero
of $V_J(x)$ at $x=\rho_\mathrm{g} +\lambda$ from the factor $[x_{j+1}-x_{j}+\mathrm{g} ]$.
\end{proof}

Specifically, by means of Lemma \ref{discretization:lem} we cast the corresponding action of $D_r$ \eqref{Dr} in terms of a discrete difference operator
in the space $\mathcal{C}(\Lambda^{(n)})$ of complex lattice functions 
$\lambda\stackrel{f}{\to} f_\lambda$:
\begin{subequations}
\begin{equation}\label{Dr:a}
(D_r  f)_\lambda =  \sum_{\substack{\lambda\subset\nu\subset\lambda+1^{n} \\ |\nu|=|\lambda |+r}}   B_{\nu/\lambda}   (\alpha,\mathrm{g};p)  f_{\nu}  
\qquad (f\in\mathcal{C}(\Lambda^{(n)}),\  \lambda\in\Lambda^{(n)} ),
\end{equation}
with
\begin{equation}\label{Dr:b}
B_{\nu/\lambda}  (\alpha,\mathrm{g};p) = \prod_{1\leq j<k\leq n}  \frac{[\lambda_j-\lambda_k+\mathrm{g}(k-j+\theta_j-\theta_k)]}{[\lambda_j-\lambda_k+\mathrm{g}(k-j)]} 
\quad\text{and}\ \theta=\nu-\lambda ,
\end{equation}
\end{subequations}
through the dictionary $f(\rho_\mathrm{g}+\lambda)=f_\lambda$, $\lambda+\varepsilon_J=\nu$ (which implies that $V_J(\rho_\mathrm{g}+\lambda)=B_{\nu/\lambda}  (\alpha,\mathrm{g};p) $).
Here and below we have employed the following standard notational conventions regarding partitions: 
\begin{equation*}
|\lambda |=\lambda_1+\cdots+\lambda_n,\qquad m^r=(\underbrace{m,\ldots,m}_r,\underbrace{0,\ldots ,0}_{n-r}),
\end{equation*}
and $\forall\lambda,\mu\in\Lambda^{(n)}$:  $\lambda\subset\mu$ iff $\lambda_j\leq\mu_j$ for $j=1,\ldots ,n$.  Since $D_r$ \eqref{Dr:a}, \eqref{Dr:b} amounts to a discretization
of the elliptic Ruijsenaars operator $D_r$ \eqref{Dr}, the commutativity is inherited automatically.

\begin{corollary}[Commutativity]\label{commutativity:cor}
The discrete Ruijsenaars operators $D_1,\ldots,D_n$ \eqref{Dr:a}, \eqref{Dr:b}  commute in $\mathcal{C}(\Lambda^{(n)})$.
\end{corollary}

\section{Eigenpolynomials}\label{sec3}

\subsection{Joint eigenfunctions}
We now define polynomials $P_\mu (\mathbf{e})=P_\mu (\mathbf{e};\alpha ,\mathrm{g};p)$, $\mu\in\Lambda^{(n)}$  in the variables
$\mathbf{e}=(\mathrm{e}_1,\ldots,\mathrm{e}_n)$ by means of the recurrence relation
\begin{subequations}
\begin{equation}\label{rec:a}
P_\mu (\mathbf{e})= \mathrm{e}_r P_\lambda (\mathbf{e})-
\sum_{\substack{\lambda\subset\nu\subset\lambda+1^{n},\,  |\nu|=|\mu| \\ \text{s.t.}\, \nu\in\Lambda^{(n)} \setminus \{ \mu\} }}   \psi^\prime_{\nu/\lambda}  (\alpha,\mathrm{g};p)
P_{\nu} (\mathbf{e})\qquad \text{if}\ \mu\neq 0,
\end{equation}
and $P_\mu (\mathbf{e})=1$ if $\mu =0$. Here $\lambda=\mu-1^r$, where
\begin{equation}\label{rec:b}
r=r_\mu=\min\{ 1\leq j\leq n \mid \mu_j-\mu_{j+1} >0\}  
\end{equation}
(with the convention $\mu_{n+1}\equiv 0$) and
\begin{equation}\label{psi}
 \psi^\prime_{\nu/\lambda} (\alpha,\mathrm{g};p) = \prod_{\substack{1\leq j<k\leq n\\ \theta_j-\theta_k=-1}} 
 {\textstyle 
  \frac{[\nu_j-\nu_k+\mathrm{g}(k-j+1)]}{[\nu_j-\nu_k+\mathrm{g}(k-j)]}
 \frac{[\lambda_j-\lambda_k+\mathrm{g}(k-j-1)]}{[\lambda_j-\lambda_k+\mathrm{g}(k-j)]}   }
\quad\text{with}\ \theta=\nu-\lambda .
\end{equation}
\end{subequations}

Let
\begin{equation}\label{degree}
d_\mu =\mu_1-\mu_n
\end{equation} 
and let $\preceq$ denote the dominance partial order on $\Lambda^{(n)}$, i.e.
\begin{equation*}
\forall \lambda,\mu\in\Lambda^{(n)}:\ \lambda\preceq\mu \Leftrightarrow  
|\lambda |=|\mu |\ \text{and}\ 
\sum_{1\leq j\leq r}\lambda_j\leq \sum_{1\leq j\leq r}\mu_j \ \text{for}\ r=1,\ldots,n-1
\end{equation*}
(and $\lambda \prec \mu$ if $\lambda\preceq\mu$ and $\lambda\neq\mu$).

\begin{proposition}[Triangularity]\label{triangularity:prp}
The polynomials $P_\mu(\mathbf{e})$, $\mu\in\Lambda^{(n)}$ are uniquely determined by the recurrence \eqref{rec:a}--\eqref{psi} (from the initial condition $P_0(\mathbf{e})=1$),
and their expansion in the monomial basis is unitriangular with respect to the
dominance order:
\begin{subequations}
\begin{equation}\label{triangular}
P_\mu (\mathbf{e}) = \mathrm{e}_\mu+  \sum_{\nu\in\Lambda^{(n)},\, \nu \prec \mu}  u_{\mu ,\nu}\,   \mathrm{e}_\nu \quad\text{with}\
u_{\mu ,\nu}= u_{\mu ,\nu}(\alpha,\mathrm{g};p)\in\mathbb{R}
\end{equation}
and 
\begin{equation}
\mathrm{e}_\mu=\prod_{1\leq j\leq n} \mathrm{e}_j^{\mu_j-\mu_{j+1}} .
\end{equation}
\end{subequations}
\end{proposition}

\begin{proof}
The proof is by lexicographical induction in $(d_\mu,r_\mu)$,  with  $r_\mu$ 
and $d_\mu$ as in  Eqs.  \eqref{rec:b} and \eqref{degree}, respectively.

If $d_\mu=0$, then either $\mu=0$ or $r=n$, i.e. $\mu = m^n$ with $m=\mu_1$. When $m=0$ we have that $\mu=0$, so $P_\mu (\mathbf{e})=\mathrm{e}_\mu=1$ by the initial condition, while for
$m>0$ the recurrence entails that
$P_\mu  (\mathbf{e})=P_{m^n} (\mathbf{e})=\mathrm{e}_nP_{(m-1)^n} (\mathbf{e})=\mathrm{e}_n^m=\mathrm{e}_\mu$.

If $d_\mu >0$,  then we have that  $\mu=\lambda +1^r $ with $r=r_\mu <n$ and  $\lambda\in\Lambda^{(n)}$.
Since $d_\lambda=d_\mu-1$, the induction hypothesis now ensures that
on the RHS of the recurrence \eqref{rec:a} the term
$\mathrm{e}_rP_\lambda  (\mathbf{e})$ expands as $\mathrm{e}_r\mathrm{e}_\lambda=\mathrm{e}_\mu$ plus a $\mathbb{R}$-linear combination of monomials of the form
$\mathrm{e}_r\mathrm{e}_\nu=\mathrm{e}_{\nu+1^r}$ with $\nu\prec\lambda$, i.e. $\nu+1^r\prec\mu$; the coefficients in this expansion stem from $P_\lambda  (\mathbf{e})$
and are thus uniquely determined from the recurrence relation (by the induction hypothesis). The  remaining terms on the RHS of the recurrence \eqref{rec:a}  consist in turn of a
$\mathbb{R}$-linear combination of $P_\nu  (\mathbf{e})$  with $d_{\nu }\leq d_\mu$ and $\nu \prec\mu$.
Moreover, one \emph{either} has that $d_{\nu}< d_\mu$ \emph{or} that $d_{\nu}= d_\mu$ with $r_\nu <r_\mu$.  Hence, in either case
the induction hypothesis guarantees that
$P_{\nu}  (\mathbf{e}) $  is uniquely determined by the recurrence relations through a
monomial expansion consisting of $\mathrm{e}_{\nu}$ perturbed by a linear combination
of $\mathrm{e}_{\tilde{\nu}}$ with $\tilde{\nu}\prec\nu$.

Upon combining all these terms appearing on the RHS of the recurrence \eqref{rec:a}, one confirms that monomial expansion of $P_\mu  (\mathbf{e})$ is of the form asserted in Eq. \eqref{triangular}
with expansion coefficients $u_{\mu,\nu}(\alpha, \mathrm{g};p)\in\mathbb{R}$ that are determined uniquely  by the recurrence relation.
\end{proof}

For $\mathbf{e}\in\mathbb{C}^n$, we  define $p(\mathbf{e})=p(\mathbf{e};\alpha,\mathrm{g};p) \in\mathcal{C}(\Lambda^{(n)})$ in terms of the
normalized polynomials
\begin{subequations}
\begin{equation}\label{p}
p_\mu(\mathbf{e})= c_\mu P_\mu (\mathbf{e}) \quad (\mu\in\Lambda^{(n)}),
\end{equation}
with
\begin{equation}\label{c}
c_\mu =c_\mu(\alpha,\mathrm{g};p)= \prod_{1\leq j<k\leq n} {\textstyle \frac{[ (k-j)\mathrm{g}]_{\mu_j-\mu_k}}{[ (k-j+1)\mathrm{g}]_{\mu_j-\mu_k}} } ,
\end{equation}
\end{subequations}
where $[z]_k$, $k=0,1,2,\ldots$ denotes the \emph{elliptic factorial}
\begin{equation*}
[z]_k=\prod_{0\leq l<k} [z+l]\quad\text{with}\  [z]_0=1.
\end{equation*}
Notice that the regularity assumption on the parameter $\mathrm{g}$ in Eq.  \eqref{g-regular} ensures that both the numerator and the denominator of $c_\mu$ do not vanish.

\begin{theorem}[Joint Eigenfunctions]\label{e-polynomials:thm}
\emph{(i)}
The functions $p(\mathbf{e})$, $\mathbf{e}\in\mathbb{C}^n$ constitute a family of joint eigenfunctions for the elliptic Ruijsenaars operators $D_1,\ldots, D_n$ \eqref{Dr:a}, \eqref{Dr:b}
in $\mathcal{C}(\Lambda^{(n)})$:
\begin{equation}\label{ev-eq:r}
D_r p(\mathbf{e})= \mathrm{e}_r p(\mathbf{e})\quad \text{for}\ r=1,\ldots ,n.
\end{equation}

\emph{(ii)} The vector of joint eigenvalues $\mathbf{e}=(\mathrm{e}_1,\ldots,\mathrm{e}_n)\in\mathbb{C}^n$ for the simultaneous eigenvalue problem in Eq. \eqref{ev-eq:r} is multiplicity free in $\mathcal{C}(\Lambda^{(n)})$.
\end{theorem}

\begin{proof}
\begin{subequations}
When evaluating at $\mu\in\Lambda^{(n)}$, the $s$th eigenvalue equation in Eq. \eqref{ev-eq:r} reads
\begin{equation}\label{eveq}
\sum_{\substack{\mu\subset\nu\subset\mu+1^{n}\\ |\nu|=|\mu | +s}}  B_{\nu/\mu}  (\alpha,\mathrm{g};p)\,
p_{\nu} (\mathbf{e})=  \mathrm{e}_s p_\mu (\mathbf{e}) .
\end{equation}
The
explicit product formulas for $B_{\mu/\lambda} (\alpha,\mathrm{g};p)$ \eqref{Dr:b}, $\psi^\prime_{\mu/\lambda} (\alpha,\mathrm{g};p)$  \eqref{psi} and $c_\mu$ \eqref{c} reveal that
for all  $\nu\in\Lambda^{(n)}$ such that $\mu\subset \nu \subset \mu+1^{n+1}$:
\begin{equation}\label{gauge}
\psi^\prime_{\nu/\mu}  (\alpha,\mathrm{g};p)  c_\mu =B_{\nu/\mu} (\alpha,\mathrm{g};p)  c_{\nu}  .
\end{equation}
Hence, Eq. \eqref{eveq} can be rewritten in terms of $P_\mu (\mathbf{e})$ as follows:
\begin{equation}\label{pieri}
\sum_{\substack{\mu\subset\nu\subset\mu+1^{n}\\ |\nu|=|\mu | +s}}  \psi^\prime_{\nu/\mu}   (\alpha,\mathrm{g};p)
P_{\nu} (\mathbf{e}) = \mathrm{e}_s P_\mu (\mathbf{e}).
\end{equation}
Since  $\psi^\prime_{\mu+1^s /\mu}  (\alpha,\mathrm{g};p)=1$, the equality in Eq. \eqref{pieri} is immediate from the recurrence  for $P_{\mu+1^s}(\mathbf{e})$ if $s\leq r_\mu$ with the convention that
$r_0=n$
(cf. Eqs. \eqref{rec:a}--\eqref{psi}).
This settles the proof of  Eq. \eqref{eveq} for $s\leq r_\mu$, but  it remains to check that the identity in question also holds
if $r_\mu<s\leq n$. To this end we perform induction with respect to the lexicographical  order on $(d_\mu,r_\mu)$ as in the proof of Proposition \ref{triangularity:prp}:

\begin{align*}
\mathrm{e}_s p_\mu(\mathbf{e}) = &c_\mu \mathrm{e}_s  P_\mu(\mathbf{e}) \\
&\stackrel{\text{Eq.}~\eqref{rec:a}}{=} c_\mu \mathrm{e}_s   \Bigl(   \mathrm{e}_r P_\lambda (\mathbf{e})-
\sum_{\substack{\lambda\subset\nu\subset\lambda+1^{n},\,  |\nu|=|\mu| \\ \text{s.t.}\, \nu\in\Lambda^{(n)} \setminus \{ \mu\} }}   \psi^\prime_{\nu/\lambda}  (\alpha,\mathrm{g};p)
P_{\nu} (\mathbf{e})                    \Bigr) \quad \intertext{(where  $r=r_\mu$ and $\lambda=\mu-1^r$) }\\
\stackrel{\text{induction}}{=}& c_\mu   \Bigl(    \mathrm{e}_s \mathrm{e}_r  P_\lambda (\mathbf{e})-
\sum_{\substack{\lambda\subset\nu\subset\lambda+1^{n},\,  |\nu|=|\mu| \\ \text{s.t.}\, \nu\in\Lambda^{(n)} \setminus \{ \mu\} }}   \psi^\prime_{\nu/\lambda}  (\alpha,\mathrm{g};p)
 c_\nu^{-1} 
  \bigl( D_s p (\mathbf{e})  \bigr)_{\nu}            \Bigr) \\
  \stackrel{\text{Eq.}~\eqref{gauge}}{=} &  \bigl( D_s p (\mathbf{e})  \bigr)_{\mu} + \frac{c_\mu }{c_\lambda}  \Bigl(    \underbrace{\mathrm{e}_s \mathrm{e}_r  p_\lambda (\mathbf{e})-
  \bigl( D_rD_s p (\mathbf{e})  \bigr)_{\lambda}  }_{=0}          \Bigr) = \bigl( D_s p (\mathbf{e})  \bigr)_{\mu} 
   \end{align*}
   as desired.
To verify the
cancellation of the underbraced terms it is essential to exploit the commutativity  of the elliptic Ruijsenaars operators: 
\begin{align*}
&\bigl( D_rD_s p (\mathbf{e})  \bigr)_{\lambda}   \stackrel{\text{Cor.}~\ref{commutativity:cor}}{=}   \bigl( D_s D_r p (\mathbf{e})  \bigr)_{\lambda}  =
\sum_{\substack{\lambda\subset\nu\subset\lambda+1^{n}\\ |\nu|=|\lambda| +s}}  B_{\nu/\lambda}  (\alpha,\mathrm{g};p)\,
 \bigl( D_r p (\mathbf{e})  \bigr)_{\nu}  \\
& \stackrel{\ast}{=}\quad
\mathrm{e}_r \sum_{\substack{\lambda\subset\nu\subset\lambda+1^{n}\\ |\nu|=|\lambda| +s}}  B_{\nu/\lambda}  (\alpha,\mathrm{g};p)\,
 p (\mathbf{e})_{\nu}  = \mathrm{e}_r  \bigl( D_s p (\mathbf{e})  \bigr)_{\lambda} \stackrel{\text{induction}}{=} \mathrm{e}_r \mathrm{e}_s p_\lambda (\mathbf{e}) .
\end{align*}
In step $\ast$ we used that either  $(d_\nu,r_\nu) < (d_\mu,r_\mu)$ in the lexicographical order (in which case the equality
$ \bigl( D_r p (\mathbf{e})  \bigr)_{\nu}=\mathrm{e}_r  p (\mathbf{e})_{\nu} $
stems from the induction hypothesis), or else $r=r_\mu\leq r_\nu$ (in which case the equality in question is plain from the recurrence relation for $P_{\nu +1^r} (\mathbf{e})$ in combination with Eq. \eqref{gauge}).
\end{subequations}

This completes the proof of part \emph{(i)} of the Theorem. Part \emph{(ii)} follows in turn from the observation that,
upon normalizing such that $p_0(\mathbf{e})=1$,
any joint eigenfunction $p(\mathbf{e})\in\mathcal{C}(\Lambda^{(n)})$ solving the eigenvalue equations in Eq. \eqref{ev-eq:r}
automatically gives rise to polynomials $P_\mu (\mathbf{e})$ \eqref{p}, \eqref{c} obeying the recurrence relation \eqref{rec:a}--\eqref{psi}.  Proposition \ref{triangularity:prp} thus implies that the joint eigenfunction is unique, i.e. the vector of
joint eigenvalues $\mathbf{e}=(\mathrm{e}_1,\ldots,\mathrm{e}_n)$ is multiplicity free in $\mathcal{C}(\Lambda^{(n)})$.
\end{proof}

\subsection{Elliptic Littlewood-Richardson coefficients}

It is clear from Proposition \ref{triangularity:prp} that the polynomials $P_\mu(\mathbf{e})$, $\mu\in \Lambda^{(n)}$ form a basis for the polynomial ring
$\mathbb{R}[\mathrm{e}_1,\ldots,\mathrm{e}_n]$. The corresponding structure constants give rise to an elliptic generalization of the Littlewood-Richardson coefficients:
\begin{equation}\label{eLR}
P_\lambda P_\mu =  \sum_{\nu\in\Lambda^{(n)}}   c^\nu_{\lambda,\mu} (\alpha, \mathrm{g};p) P_\nu 
\end{equation}
(where the arguments $\mathbf{e}$ are suppressed). For $\mu=1^r$, an explicit product formula for these elliptic Littlewood-Richardson coefficients is immediate from
Theorem \ref{e-polynomials:thm}.

\begin{corollary}[Pieri Rule]\label{pieri:cor}
For $\mu\in\Lambda^{(n)}$ and $1\leq r\leq n$, one has that
\begin{equation}\label{pieri:r}
P_\lambda P_{1^r}  = 
\sum_{\substack{\lambda\subset\nu\subset\lambda+1^{n} \\ |\nu|=|\lambda |+r }}   \psi^\prime_{\nu/\lambda}  (\alpha,\mathrm{g};p)
P_{\nu} .
\end{equation}
\end{corollary}

\begin{proof}
Since $P_{1^r}(\mathbf{e})=\mathrm{e}_{1^r}=\mathrm{e}_r$, the asserted Pieri rules encode the eigenvalue equations of Theorem \ref{e-polynomials:thm}
in the reformulation of Eq. \eqref{pieri} (obtained via Eq. \eqref{gauge}).
\end{proof}

With the aid of the Pieri rule, it is readily seen that the classical Littlewood-Richardson coefficients $c_{\lambda ,\mu}^\nu$ for the Schur polynomials
 \cite[Chapter I.9]{mac:symmetric}  and their  two-parameter deformation $f_{\lambda ,\mu}^\nu(q,t)$
associated with the Macdonald polynomials  \cite[Chapter VI.7]{mac:symmetric} arise as suitable parameter specializations of the three-parameter elliptic Littlewood-Richardson coefficients
$ c^\nu_{\lambda,\mu}  (\alpha, \mathrm{g};p)  $ \eqref{eLR}.

\begin{proposition}[Degenerations]\label{LRdegenerations:prp}
The classical Littlewood-Richardson coefficients $c_{\lambda ,\mu}^\nu$ and Macdonald's $(q,t)$-Littlewood-Richardson coefficients  $f_{\lambda ,\mu}^\nu(q,t)$ are recovered
from $ c^\nu_{\lambda,\mu}  (\alpha, \mathrm{g};p)  $ \eqref{eLR}  in the following way:
\begin{subequations}
\begin{equation}\label{g=1:LR}
\lim_{\mathrm{g}\to 1}  c^\nu_{\lambda,\mu}  (\alpha, \mathrm{g};p)  = c_{\lambda ,\mu}^\nu
\end{equation}
(provided $\frac{2\pi}{\alpha}>0$ is irrational), and
\begin{equation}\label{p=0:LR}
\lim_{p\to 0}  c^\nu_{\lambda,\mu}  (\alpha, \mathrm{g};p)  = f_{\lambda ,\mu}^\nu(q,q^{\mathrm{g}})
 \quad\text{with}\quad  q=e^{\text{i}\alpha}  
\end{equation}
\end{subequations}
(provided $j\mathrm{g}\not\in\mathbb{Z}_{\leq 0}+ \frac{2\pi}{\alpha}\mathbb{Z}$ for $j=1,\ldots ,n$).

In particular, for $\lambda,\nu\in\Lambda^{(n)}$ such that $\lambda \subset \nu \subset \lambda +1^n$ one has with these genericity assumptions in place that
\begin{subequations}
\begin{equation}\label{g=1:pieri}
\lim_{\mathrm{g}\to 1} \psi^\prime_{\nu/\lambda}   (\alpha,\mathrm{g};p)= 1
\end{equation}
and
\begin{equation}\label{p=0:pieri}
\lim_{p\to 0} \psi^\prime_{\nu/\lambda}  (\alpha,\mathrm{g};p) =  \prod_{\substack{1\leq j<k\leq n\\ \theta_j-\theta_k=-1}} 
 {\textstyle 
  \frac{[\nu_j-\nu_k+\mathrm{g}(k-j+1)]_q}{[\nu_j-\nu_k+\mathrm{g}(k-j)]_q}
 \frac{[\lambda_j-\lambda_k+\mathrm{g}(k-j-1)]_q}{[\lambda_j-\lambda_k+\mathrm{g}(k-j)]_q}   }  ,
\end{equation}
where $\theta=\nu-\lambda$ and
\end{subequations}
\begin{equation*}
{\textstyle [z]_q=\frac{\sin( \frac{\alpha z}{2})}{\sin(\frac{\alpha}{2})}= \frac{q^{\frac{z}{2}}-q^{-\frac{z}{2}}}{q^{\frac{1}{2}}-q^{-\frac{1}{2}}}  }.
\end{equation*}
\end{proposition}

\begin{proof}
By definition, the polynomial ring $\mathbb{R}[\mathrm{e}_1,\ldots,\mathrm{e}_n]$ is generated by the monomials $\mathrm{e}_r=P_{1^r}(\mathbf{e})$, $r=1,\ldots ,n$. It is therefore
sufficient to verify the limits
in Eqs. \eqref{g=1:LR},  \eqref{p=0:LR} for $\lambda=1^r$ ($r=1,\ldots ,n$), which---by the Pieri rule of Corollary \ref{pieri:cor}---amounts to checking the limits in Eqs. \eqref{g=1:pieri}, \eqref{p=0:pieri}. The genericity assumptions on the parameters ensure that none of the denominators vanish.

Specifically, in the product formula for $ \psi^\prime_{\nu/\lambda}$ \eqref{psi} one has that $\lambda_j-\lambda_k=\nu_j-\nu_k+1$ if $\theta_j-\theta_k=-1$, so
the limit in Eq. \eqref{g=1:pieri} is evident.   We thus recover in this manner  the (dual) Pieri rules for the Schur polynomials $s_\mu (x)$ \cite[Ch. I, Eq. (5.17)]{mac:symmetric}
from Corollary \ref{pieri:cor}, which in turn implies the limit in Eq. \eqref{g=1:LR}.

Similarly, since $[z;p]$ \eqref{scaled-theta}, \eqref{theta} extends analytically to $-1<p<1$ with $[z;0]=\frac{\alpha}{2}\sin( \frac{\alpha z}{2})$, the limit in Eq. \eqref{p=0:pieri} is also manifest from Eq. \eqref{psi}. Upon comparing with \cite[Ch. VI, Eqs. (6.7${}^\prime$), (6.13)]{mac:symmetric}, we see that at $p=0$  the Pieri rules for the Macdonald polynomials
$P_\mu (x;q,q^{\mathrm{g}})$ are recovered from Corollary \ref{pieri:cor}, therewith settling the limit in Eq. \eqref{p=0:LR}.
\end{proof}

The Pieri rule also confirms that some well-known vanishing properties enjoyed by the Littlewood-Richardson coefficients persist at the elliptic level.

\begin{proposition}[Vanishing Terms]\label{vanishing:prp}
For $\lambda,\mu,\nu\in\Lambda^{(n)}$,
the  elliptic Littlewood-Richardson coefficient  $c^\nu_{\lambda,\mu} (\alpha, \mathrm{g};p)$ vanishes unless
 $\lambda\subset \nu$ and  $\mu\subset \nu$ with $|\lambda |+ |\mu |=|\nu|$.
\end{proposition}

\begin{proof}
It suffices to mimic the proof for the corresponding statement at $p=0$ from \cite[Ch. VI, Eq. (7.4)]{mac:symmetric}.
Since the monomial expansion of $P_\mu (\mathbf{e})$ in Proposition \ref{triangularity:prp} involves only monomials $\mathrm{e}_\kappa$ with $|\kappa |=|\mu|$ and
$\mathrm{e}_\kappa \mathrm{e}_{\tilde{\kappa}}=\mathrm{e}_{\kappa+\tilde{\kappa}}$, it is clear that
$c^\nu_{\lambda,\mu} (\alpha, \mathrm{g};p)$ \eqref{eLR} can only be nonzero provided $|\nu |=|\lambda|+|\mu|$. Moreover, let $\mathcal{I}_\lambda$ denote the subspace of $\mathbb{R}[\mathrm{e}_1,\ldots ,\mathrm{e}_n]$ spanned by the $P_\kappa (\mathbf{e})$ with $\lambda\subset\kappa$. It is manifest from the Pieri rule that
$\mathrm{e}_r \mathcal{I}_\lambda\subset \mathcal{I}_\lambda$ for $r=1,\ldots,n$, so $\mathcal{I}_\lambda$ is an ideal in $\mathbb{R}[\mathrm{e}_1,\ldots ,\mathrm{e}_n]$. It thus follows that $P_\lambda(\mathbf{e})P_\mu(\mathbf{e})\in \mathcal{I}_\lambda\cap \mathcal{I}_\mu$.
\end{proof}

\subsection{Symmetric polynomials}
Let us recall (cf. the proof of Proposition \ref{LRdegenerations:prp}) that $s_\mu (x)=s_\mu (x_1,\ldots ,x_n)$ and $P_\mu(x;q,t)=P_\mu(x_1,\ldots ,x_n;q,t)$ refer  to the  Schur polynomials \cite[Chapter I]{mac:symmetric} and the
Macdonald polynomials \cite[Chapter VI]{mac:symmetric}, respectively. The symmetric  polynomials in question are monic in the sense that their leading monomial  is given by
\begin{equation}\label{msf}
m_\mu (x)  =m_\mu (x_1,\ldots ,x_n)  = \sum_{\nu\in S_{n}(\mu)}  x_1^{\nu_1}\cdots x_{n}^{\nu_{n}}\qquad (\mu\in\Lambda^{(n)}),
\end{equation}
where the sum is over all compositions reordering the parts of $\mu$ (i.e. over the orbit of $\mu$ with respect to the action of the 
permutation-group $S_{n}$ of permutations
$\sigma= { \bigl( \begin{smallmatrix}1& 2& \cdots & n \\
 \sigma_1&\sigma_2&\cdots & \sigma_{n}
 \end{smallmatrix}\bigr)}$ on $\mu_1,\mu_2,\ldots,\mu_{n}$).  It is instructive to describe the precise relation between our elliptic eigenpolynomials and these two standard bases for the
ring $\mathcal{A}^{(n)}=\mathbb{R}[x_1,\ldots ,x_n]^{S_n}$ of symmetric polynomials  in the variables $x_1,\ldots, x_n$.
 
To this end, let us observe that the polynomials $P_\mu (\mathbf{e})$, $\mu\in\Lambda^{(n)}$ give rise
 to a (monic) basis for $\mathcal{A}^{(n)}$ through the ring isomorphism $\mathbb{R}[\mathrm{e}_1,\ldots,\mathrm{e}_n]\cong \mathcal{A}^{(n)}$
determined by the injection
$\mathrm{e}_r\to m_{1^r}(x)$,  $r=1,\ldots,n$:
\begin{equation}\label{eMP}
R_\mu (x;\alpha,\mathrm{g}; p) = P_\mu ( \mathbf{e})\quad \text{with}\ \mathbf{e}=\bigl( m_{1^1}(x), m_{1^2}(x),\ldots,m_{1^n}(x)              \bigr) .
\end{equation}
From the limits in Proposition \ref{LRdegenerations:prp}, it is then clear that the corresponding degenerations of the recurrence relations \eqref{rec:a}--\eqref{psi} reproduce the
Schur polynomials and the Macdonald polynomials respectively (cf.  \cite[Ch. I, Eq. (5.17)]{mac:symmetric} and
\cite[Ch. VI, Eqs. (6.7${}^\prime$), (6.13)]{mac:symmetric}).

\begin{corollary}[The Schur and Macdonald Limits]\label{Pdegenerations:cor}
For any $\mu\in\Lambda^{(n)}$, 
the Schur polynomial $s_\mu (x)$ and the Macdonald polynomial $P_\mu (x;q,q^\mathrm{g})$ are recovered
from $R_\mu  (x;\alpha, \mathrm{g};p)  $ \eqref{eMP}  in the following way:
\begin{subequations}
\begin{equation}\label{g=1:P}
\lim_{\mathrm{g}\to 1} R_\mu (x;\alpha, \mathrm{g};p)  = s_\mu (x)
\end{equation}
(provided $\frac{2\pi}{\alpha}>0$ is irrational), and
\begin{equation}\label{p=0:P}
\lim_{p\to 0} R_\mu  (x;\alpha, \mathrm{g};p)  =P_\mu (x; q,q^{\mathrm{g}})
 \quad\text{with}\quad  q=e^{\text{i}\alpha}  
\end{equation}
(provided $j\mathrm{g}\not\in\mathbb{Z}_{\leq 0}+ \frac{2\pi}{\alpha}\mathbb{Z}$ for $j=1,\ldots ,n$).
\end{subequations}
\end{corollary}

\section{Generic elliptic deformation of the fusion ring for $\widehat{\mathfrak{su}}(n)_m$}\label{sec4}

\subsection{Character ring for $\mathfrak{su}(n)$}
When $r=n$ the Pieri rule \eqref{pieri:r} simply states that
$P_{\mu +1^n}(\mathbf{e})=P_{\mu}(\mathbf{e}) P_{1^n}(\mathbf{e}) =P_{\mu}(\mathbf{e}) \mathrm{e}_n$. Hence, one has more generally that
\begin{equation}\label{t-symmetry}
P_{\mu}(\mathbf{e}) = P_{\underline{\mu}}(\mathbf{e}) \mathrm{e}_n^{\mu_n}\quad
\text{with}\ 
\underline{\mu}=(\mu_1-\mu_{n},\mu_2-\mu_{n},\ldots,\mu_{n-1}-\mu_{n} ,0). 
\end{equation}
To divide out this translational symmetry one substitutes $\mathrm{e}_n=1$, therewith reducing to  the  ring
\begin{subequations}
\begin{equation}\label{R0}
\mathcal{R}^{(n)}_0=\mathbb{R}[\mathrm{e}_1,\ldots ,\mathrm{e}_n]/\langle \mathrm{e}_n-1\rangle  
\end{equation}
for which (the corresponding specialization of) the polynomials
\begin{equation}
P_\mu(\mathbf{e})  \quad \text{with}\ \mu\in
\Lambda^{(n)}_0=\{ \lambda\in \Lambda^{(n)} \mid \lambda_n=0 \} 
\end{equation}
\end{subequations}
provide a basis. In view of Proposition \ref{vanishing:prp}, it is immediate from Eq. \eqref{t-symmetry} that
the corresponding structure constants for $\mathcal{R}^{(n)}_0$  can be expressed in terms of elliptic Littlewood-Richardson coefficients
$c^\nu_{\lambda,\mu} (\alpha, \mathrm{g};p) $ \eqref{eLR} as follows:
\begin{subequations}
\begin{equation}\label{esl-LR}
P_\lambda P_\mu =  \sum_{\substack{  \nu\supset \lambda,\ \nu\supset\mu \\ |\nu |= |\lambda | +|\mu |     }}  c^\nu_{\lambda,\mu} (\alpha, \mathrm{g};p) P_{\underline{\nu} }
\qquad (\lambda ,\mu\in \Lambda^{(n)}_0,\, \nu\in\Lambda^{(n)} ).
\end{equation}
In particular, for  $1\leq r< n$ and $\lambda\in\Lambda^{(n)}_0$ one retrieves from Corollary \ref{pieri:cor} that
\begin{equation}\label{esl-pieri:r}
 P_\lambda  P_{1^r} = 
\sum_{\substack{\lambda\subset\nu\subset\lambda+1^{n} \\ |\nu|=|\lambda |+r }}   \psi^\prime_{\nu/\lambda}  (\alpha,\mathrm{g};p)
P_{\underline{\nu}} .
\end{equation}
\end{subequations}

Proposition \ref{LRdegenerations:prp} and Corollary \ref{Pdegenerations:cor} entail that  at $p=0$
Eq. \eqref{esl-LR} recovers the structure constants for the multiplication in the basis of Macdonald polynomials associated with
(the root system of) the complex simple Lie algebra $\mathfrak{su}(n)$ \cite{mac:orthogonal,mac:affine}. Indeed, Eq. \eqref{esl-pieri:r} degenerates to the Pieri rule
for the $\mathfrak{su}(n)$ Macdonald polynomials in this situation (which is obtained from the Pieri formula for  $P_\mu(x_1,\ldots,x_n;q,q^{\mathrm{g}})$ in
 \cite[Ch. VI, Eqs. (6.7${}^\prime$), (6.13)]{mac:symmetric} by dividing out the ideal generated by $m_{1^n}(x_1,\ldots ,x_n)-1$). Here 
partitions $\lambda\in\Lambda^{(n)}_0$ are identified with dominant weight vectors for $\mathfrak{su}(n)$ in the standard way:
\begin{equation}\label{pw-bijection}
\lambda\leftrightarrow \sum_{1\leq r <n} (\lambda_r-\lambda_{r+1})\varpi_r ,
\end{equation}
where
$\varpi_r\leftrightarrow 1^r$, $r=1,\ldots,n-1$ refers to the corresponding basis of fundamental weight vectors (labeled in accordance with the plates of \cite{bou:groupes}).

Similarly, for
$\mathrm{g}\to 1$  Eq. \eqref{esl-LR} encodes the structure constants for the character ring of $\mathfrak{su}(n)$  in the basis of the irreducible characters. For instance,
in this limit the Pieri rule \eqref{esl-pieri:r}  counts  the
tensor multiplicities for tensoring with a fundamental representation, cf. e.g.  \cite[Chapter 9.10]{pro:lie}.

\subsection{Fusion ideal}
We now scale $\alpha$ in terms of  $\mathrm{g}\in \mathbb{R}\setminus\mathbb{Q}$ as follows:
\begin{equation}\label{tc}
\boxed{\alpha=\frac{2\pi}{m+n\mathrm{g}} \quad \text{with}\ m\in\mathbb{N}.} 
\end{equation}
The irrationality of $\mathrm{g}$ then guarantees that the regularity requirement in Eq.
 \eqref{denom-regular} is satisfied, so the polynomials $P_\mu(\mathbf{e})$ \eqref{rec:a}--\eqref{psi} are well-defined for this parameter specialization.
Let us consider the following ideal in $\mathcal{R}^{(n)}_0$ \eqref{R0}:
\begin{equation}\label{Inm}
\mathcal{I}^{(n,m)}= \langle  P_\mu (\mathbf{e}) \mid \mu\in\Lambda^{(n)}_0\,\text{with}\, d_\mu = m+1   \rangle  .
\end{equation}
This ideal should be viewed as an elliptic $(\mathrm{g},p)$-deformation associated with the elliptic Ruijsenaars model of the fusion ideal for $\widehat{\mathfrak{su}}(n)_m$ Wess-Zumino-Witten
conformal field theories (cf. e.g. \cite{dif-mat-sen:conformal,fuc:affine,gep:fusion,goo-wen:littlewood,kor-str:slnk} and references therein).

With the aid of the  Pieri rule \eqref{esl-pieri:r} and following lemma, one arrives at
a convenient a basis for $\mathcal{I}^{(n,m)}$ in terms of the  eigenpolynomials associated with the elliptic Ruijsenaars lattice model.

\begin{lemma}[Level $m$ Boundary Condition]\label{vanishing:lem}
Let $\lambda\in\Lambda^{(n)}_0$ with $d_\lambda=m+1$ and let $\lambda\subset\nu\subset\lambda+1^n$.
Then one has that
\begin{equation}
 \psi^\prime_{\nu/\lambda}  ({\textstyle \frac{2\pi}{m+n\mathrm{g}}},\mathrm{g};p)=0\quad \text{if}\ d_\nu \leq m
\end{equation}
(assuming $\mathrm{g}\in \mathbb{R}\setminus\mathbb{Q}$).
\end{lemma}

\begin{proof}
The conditions imply that $\nu=\lambda+\theta$ with $\theta$ a vertical $r$-strip ($1\leq r< n$), so 
$d_\nu=d_\lambda+\theta_1-\theta_n$ with $\theta_1,\theta_n\in \{ 0,1\}$. If $d_\nu\leq m$, we must in fact have that
$\theta_1=0$, $\theta_n=1$ and $d_\nu=m$ as $d_\lambda=m+1$.  In this situation,
$\psi^\prime_{\nu/\lambda}  ({\textstyle \frac{2\pi}{m+n\mathrm{g}}},\mathrm{g};p)$ \eqref{psi} picks up a zero
from the factor $[\nu_j-\nu_k+(k-j+1)\mathrm{g}]$ in the numerator for $j=1$ and $k=n$:
$[\nu_1-\nu_n+n\mathrm{g}]=[ m+n\mathrm{g}] =[{\textstyle \frac{2\pi}{\alpha}}]=0.$
\end{proof}

\begin{proposition}[Basis for $\mathcal{I}^{(n,m)}$]\label{I-basis:prp}
For $\alpha=\frac{2\pi}{m+n\mathrm{g}}$ and $\mathrm{g}\in \mathbb{R}\setminus\mathbb{Q}$,
the polynomials $P_\mu (\mathbf{e}) $ with $\mu\in\Lambda^{(n)}_0$ such that $d_\mu>m$ constitute a basis for $\mathcal{I}^{(n,m)}$.
\end{proposition}

\begin{proof}
By definition, the ideal $\mathcal{I}^{(n,m)}$ \eqref{Inm} consists of polynomials of the form
\begin{equation}\label{ideal}
\sum_{\lambda\in\Lambda^{(n)}_0,\, d_\lambda =m+1} a_\lambda(\mathbf{e}) P_\lambda (\mathbf{e})\quad \text{with}\ a_\lambda (\mathbf{e})\in \mathcal{R}^{(n)}_0.
\end{equation}
Since the monomials $\mathrm{e}_r=P_{1^r}(\mathbf{e})$ ($1\leq r<n$) generate $\mathcal{R}^{(n)}_0$, it is immediate from
the Pieri rule \eqref{esl-pieri:r} and Lemma \ref{vanishing:lem} that all products $a_\lambda(\mathbf{e}) P_\lambda (\mathbf{e})$ in the sum of Eq. \eqref{ideal} expand as  $\mathbb{R}$-linear combinations of
basis polynomials $P_\mu(\mathbf{e})$ with $\mu\in\Lambda^{(n)}_0$ and $d_\mu >m$. It remains to check that the pertinent basis polynomials $P_\mu(\mathbf{e})$ indeed belong to $\mathcal{I}^{(n,m)}$.
If $d_\mu=m+1$ this is the case by definition, while for $d_\mu>m+1$ it follows by lexicographical induction in $(d_\mu,r_\mu)$ from the recurrence in Eqs. \eqref{rec:a}--\eqref{psi} 
with the aid of Lemma \ref{vanishing:lem} (using also that $P_\nu (\mathbf{e})=P_{\underline{\nu}} (\mathbf{e})$ in $\mathcal{R}^{(n)}_0$). Indeed, on the RHS of \eqref{rec:a}
one has that $d_\lambda=d_\mu-1> m$, so  $\mathrm{e}_rP_\lambda(\mathbf{e})\in\mathcal{I}^{(n,m)}$ since $P_\lambda(\mathbf{e})\in\mathcal{I}^{(n,m)}$ by virtue of the induction hypothesis. The remaining terms on the RHS of Eq. \eqref{rec:a} involve polynomials $P_\nu(\mathbf{e})=P_{\underline{\nu}}(\mathbf{e})$ with
\emph{either} $d_\nu=d_\mu>m$ and $r_\nu <r_\mu$ \emph{or} (using Lemma \ref{vanishing:lem}) with $d_\mu>d_\nu >m $; the induction hypothesis therefore guarantees again that all of these terms belong to $\mathcal{I}^{(n,m)}$. We may thus conclude that $P_\mu (\mathbf{e})$ \eqref{rec:a} lies in  $\mathcal{I}^{(n,m)}$, therewith completing the induction step.
\end{proof}

\subsection{Fusion ring}
Upon dividing out  $\mathcal{I}^{(n,m)}$, one arrives in turn at a corresponding elliptic deformation of the fusion ring for $\widehat{\mathfrak{su}}(n)_m$ Wess-Zumino-Witten
conformal field theories:
\begin{equation}\label{f-ring}
\mathcal{R}^{(n,m)}_0=\mathcal{R}^{(n)}_0/ \mathcal{I}^{(n,m)} 
\end{equation}
 (cf.  again \cite{dif-mat-sen:conformal,fuc:affine,gep:fusion,goo-wen:littlewood,kor-str:slnk} and references therein).
For $P\in \mathcal{R}^{(n)}_0$, we will denote its coset $P+\mathcal{I}^{(n,m)}$ in $\mathcal{R}^{(n,m)}_0$ by $[P]$.
The cosets of the elliptic  eigenpolynomials labeled by bounded partitions in
\begin{equation}\label{dcone}
 \Lambda_0^{(n,m)}=  \{  \lambda\in \Lambda^{(n)} _0\mid d_\lambda \leq m \} 
\end{equation}
provide a basis for our elliptic fusion ring. Notice in this connection that the bijection \eqref{pw-bijection} maps the bounded partitions in question to (the nonaffine parts of) the dominant weights of
the affine Lie algebra $\widehat{\mathfrak{su}}(n)_m$ (cf. e.g. \cite[Section 2.1]{kor-str:slnk}).

\begin{proposition}[Basis for $\mathcal{R}^{(n,m)}_0$]\label{R-basis:prp}
For $\alpha=\frac{2\pi}{m+n\mathrm{g}}$ and $\mathrm{g}\in \mathbb{R}\setminus\mathbb{Q}$,
the cosets $[P_\mu (\mathbf{e}) ] $, $\mu\in\Lambda^{(n,m)}_0$ constitute a basis for $\mathcal{R}^{(n,m)}_0$.
\end{proposition}

\begin{proof}
Since the polynomials $P_\mu(\mathbf{e})$, $\mu\in\Lambda^{(n)}_0$ form a basis for $\mathcal{R}^{(n)}_0$, the assertion in the proposition is immediate from Proposition \ref{I-basis:prp}. Indeed, the kernel of the ring homomorphism $P\to [P]$ from
$\mathcal{R}^{(n)}_0$ onto $\mathcal{R}^{(n,m)}_0$ is equal to the  $\mathbb{R}$-span of the polynomials
$P_\mu(\mathbf{e})$, $\mu\in\Lambda^{(n)}_0\setminus\Lambda^{(n,m)}_0$ (by Proposition \ref{I-basis:prp}), so the cosets 
$[P_\mu(\mathbf{e})]$, $\mu\in\Lambda^{(n,m)}_0$ constitute a basis for $\mathcal{R}^{(n,m)}_0$.
\end{proof}

It is now straightforward to express the structure constants of $\mathcal{R}^{(n,m)}_0$ in the basis $[P_\mu(\mathbf{e})]$, $\mu\in\Lambda^{(n,m)}_0$ in terms
of elliptic Littlewood-Richardson coefficients.

\begin{theorem}[Structure Constants of $\mathcal{R}^{(n,m)}_0$]\label{e-ring:thm}
For $\mathrm{g}\in\mathbb{R}\setminus\mathbb{Q}$, 
the structure constants of $\mathcal{R}^{(n,m)}_0$ in the basis $[P_\mu(\mathbf{e})]$, $\mu\in\Lambda^{(n,m)}_0$ can be expressed in terms of elliptic Littlewood-Richardson coefficients as follows:
\begin{subequations}
\begin{equation}\label{eslm-LR}
[P_\lambda ]  [ P_\mu ] =  \sum_{\substack{  \nu\supset \lambda,\ \nu\supset\mu \\ |\nu |= |\lambda | +|\mu | ,\, d_\nu \leq m    }} 
c^\nu_{\lambda,\mu} \bigl({\textstyle \frac{2\pi}{m+n\mathrm{g}} }, \mathrm{g};p\bigr) [P_{\underline{\nu} } ]
\qquad (\lambda ,\mu\in \Lambda^{(n,m)}_0,\, \nu\in\Lambda^{(n)} ).
\end{equation}
In particular, for  $1\leq r< n$ and $\lambda\in\Lambda^{(n,m)}_0$ one has explicitly that
\begin{equation}\label{eslm-pieri:r}
[P_\lambda]  [P_{1^r} ]  = 
\sum_{\substack{\lambda\subset\nu\subset\lambda+1^{n} \\ |\nu|=|\lambda |+r ,\, d_\nu \leq m}}   \psi^\prime_{\nu/\lambda}  \bigl({\textstyle \frac{2\pi}{m+n\mathrm{g}} }, \mathrm{g};p\bigr)
[P_{\underline{\nu}}  ] .
\end{equation}
\end{subequations}
\end{theorem}

\begin{proof}
In view of Propositions \ref{I-basis:prp} and \ref{R-basis:prp},
the asserted formulas in Eqs. \eqref{eslm-LR} and \eqref{eslm-pieri:r} follow from Eqs. \eqref{esl-LR} and \eqref{esl-pieri:r} upon applying the ring homomorphism $P\to [P]$ from
$\mathcal{R}^{(n)}_0$ onto $\mathcal{R}^{(n,m)}_0$.
\end{proof}

\section{$\widehat{\mathfrak{su}}(n)_m$ Verlinde algebras from elliptic Ruijsenaars systems}\label{sec5}

\subsection{Analyticity}
For $\mathrm{g}>0$ and $\mu\in\Lambda^{(n)}_0$ with $d_\mu\leq m+1$, we now employ analytic continuation in the parameters of
$P_\mu \bigl(\mathbf{e};\frac{2\pi}{m+n\mathrm{g}},\mathrm{g};p\bigr)\in \mathcal{R}^{(n)}_0$ so as to remove the restriction that $\mathrm{g}$ be irrational.

\begin{lemma}[Positivity of the Recurrence Coefficients]\label{regularity:lem}
Let $\alpha=\frac{2\pi}{m+n\mathrm{g}}$,  $\mathrm{g}>0$, $-1<p<1$, and $\nu \in\Lambda^{(n)}$.

\emph{(i)}  For $1\leq j< k\leq n$, one has that
$[\nu_j-\nu_k+(k-j)\mathrm{g};p]  > 0$ if $\nu_j-\nu_k \leq m$.

\emph{(ii)} For $\lambda\in\Lambda^{(n,m)}_0$ and $\lambda\subset\nu\subset\lambda+1^n$, the recurrence coefficient
$ \psi^\prime_{\nu/\lambda} (\frac{2\pi}{m+n\mathrm{g}},\mathrm{g};p)$ \eqref{psi} is analytic in $\mathrm{g}\in (0,\infty)$ and positive.
\end{lemma}

 \begin{proof}
 \emph{(i)} The asserted positivity  of the theta factor is plain from the product expansion  \eqref{theta} via the estimate
 \begin{equation*}
 0<\mathrm{g}\leq \nu_j-\nu_k+(k-j)\mathrm{g} < m+n\mathrm{g} =\frac{2\pi}{\alpha} .
 \end{equation*}

\emph{(ii)} The analyticity of  $\psi^\prime_{\nu/\lambda} (\frac{2\pi}{m+n\mathrm{g}},\mathrm{g};p)$ in $\mathrm{g}>0$  follows in turn from the fact that the zeros in the denominators are avoided as a consequence of part \emph{(i)}.  Indeed, in all factors of $\psi^\prime_{\nu/\lambda} (\frac{2\pi}{m+n\mathrm{g}},\mathrm{g};p)$ \eqref{psi} the denominators stay positive because
$\lambda_j-\lambda_k\leq d_\lambda\leq m$ and $\nu_j-\nu_k=\lambda_j-\lambda_k+\theta_j-\theta_k\leq m-1$. The positivity of the corresponding numerators
is subsequently seen from the estimates 
\begin{equation*}
0\leq\nu_j-\nu_k+(k-j-1)\mathrm{g} < \lambda_j-\lambda_k+(k-j-1)\mathrm{g} < m+n\mathrm{g} 
\end{equation*}
and $0< \nu_j-\nu_k+(k-j+1)\mathrm{g} < m+n\mathrm{g} $.
 \end{proof}

 \begin{proposition}[Analyticity of Basis Polynomials]\label{P-reg:prp}
For $\mu \in\Lambda^{(n)}_0$ the polynomial $P_\mu \bigl( \mathbf{e};\frac{2\pi}{m+n\mathrm{g}},\mathrm{g};p\bigr)\in \mathcal{R}^{(n)}_0$
 is analytic, both in $\mathrm{g}\in (0,\infty)$ and in $p\in (-1,1)$, provided $d_\mu\leq m+1$.
 \end{proposition}
 
 \begin{proof}
 By  Lemma \ref{regularity:lem},
 the conditions make sure that all coefficients $ \psi^\prime_{\nu/\lambda} (\frac{2\pi}{m+n\mathrm{g}},\mathrm{g};p)$ in the recurrence \eqref{rec:a}--\eqref{psi} for
 $P_\mu \bigl( \mathbf{e};\frac{2\pi}{m+n\mathrm{g}},\mathrm{g};p\bigr)$ enjoy the asserted analyticity, which is therefore inherited by the polynomials in question.
  \end{proof}
 
 For later reference, let us also explicitly check that the pertinent normalization coefficients $c_\mu(\alpha,\mathrm{g};p)$ from Eq. \eqref{c} permit analytic continuation to $\mathrm{g}>0$.
 
 \begin{lemma}[Positivity of the Normalization Constants]\label{c-reg:lem}
 For $\mu\in\Lambda^{(n,m)}_0$, $\mathrm{g}>0$ and $-1<p<1$, the normalization coefficient
 $c_\mu(\frac{2\pi}{m+n\mathrm{g}},\mathrm{g};p)$ \eqref{c} is \emph{positive}.
  \end{lemma}
 
 \begin{proof}
 The arguments of the theta functions in the numerator and the denominator of  $c_\mu(\frac{2\pi}{m+n\mathrm{g}},\mathrm{g};p)$
 are of the form $l+k\mathrm{g}$, with $0\leq l <m$ and $1\leq k \leq n$, so
 $0<l+k\mathrm{g}< m+n\mathrm{g}=\frac{2\pi}{\alpha}$.  The corresponding values of the theta function are thus positive by the product formula \eqref{theta}.
 \end{proof}

\subsection{Spectral variety}

In \cite{die-gor:elliptic} it was shown that for $\alpha$ \eqref{tc} with $\mathrm{g}>0$, the
elliptic Ruijsenaars operators truncate to commuting discrete difference operators of the form (cf. Eqs. \eqref{Dr:a}, \eqref{Dr:b}):
\begin{equation}\label{Drf}
(D_r  f)_\lambda =  \sum_{\substack{\lambda\subset\nu\subset\lambda+1^{n} \\ |\nu|=|\lambda |+r ,\, d_\nu \leq m}} 
B_{\nu/\lambda}  \bigl( {\textstyle \frac{2\pi}{m+n\mathrm{g}}},\mathrm{g};p\bigr)  f_{\underline\nu} ,\quad 1\leq r <n.
\end{equation}
These operators turn out to be normal  in the Hilbert space $\ell^2 (\Lambda^{(n,m)}_0,\Delta)$ with the inner product
\begin{subequations}
\begin{equation}\label{ip}
\langle f,g\rangle_\Delta=\sum_{\lambda\in\Lambda_0^{(n,m)}}  f_\lambda \overline{ g_\lambda }\Delta_\lambda
\qquad \bigl( f,g\in \ell^2(\Lambda^{(n,m)}_0,\Delta ) \bigr),
\end{equation}
where 
\begin{equation}\label{delta}
\Delta_\lambda =\Delta_\lambda (\alpha,\mathrm{g};p)=
\prod_{1\leq j<k\leq n}
{\textstyle
\frac{[\lambda_j-\lambda_k+(k-j)\mathrm{g}]}{[(k-j)\mathrm{g} ]}\frac{[(k-j+1)\mathrm{g} ]_{\lambda_j-\lambda_k}}{[1+(k-j-1)\mathrm{g} ]_{\lambda_j-\lambda_k}} } 
\end{equation}
\end{subequations}
(cf.  \cite[Proposition 6]{die-gor:elliptic}). Their joint spectrum is moreover given by $\binom{n-1+m}{m}$ multiplicity-free vectors (cf.  \cite[Corollary 10]{die-gor:elliptic})
\begin{subequations}
\begin{equation}\label{evR:a}
\mathbf{e}_\nu = \mathbf{e}_\nu  \bigl( {\textstyle \frac{2\pi}{m+n\mathrm{g}}},\mathrm{g};p\bigr) =\bigl( \mathrm{e}_{1,\nu},\ldots ,\mathrm{e}_{n-1,\nu},1\bigr)\in\mathbb{C}^n\quad  (\nu\in\Lambda_0^{(n,m)})
\end{equation}
 such that
\begin{equation}\label{evR:b}
D_r p(\mathbf{e}_\nu) = \mathrm{e}_{r,\nu}   p(\mathbf{e}_\nu)\quad (1\leq r <n,\, \nu\in\Lambda_0^{(n,m)}) ,
\end{equation}
which are analytic in $p\in (-1,1)$ with
\begin{align}\label{evR:c}
\lim_{p\to 0} & \mathrm{e}_{r,\nu}  \bigl( {\textstyle \frac{2\pi}{m+n\mathrm{g}}},\mathrm{g};p\bigr)  = \\
&q^{-r \bigl( \frac{|\nu|}{n}+\frac{(n-1)\mathrm{g}}{2}\bigr)}
m_{1^r} (q^{\nu_1+(n-1)\mathrm{g}},q^{\nu_2+(n-2)\mathrm{g}},\ldots, q^{\nu_{n-1}+\mathrm{g}},1) . \nonumber
\end{align}
\end{subequations}
Here  $p(\mathbf{e})$ is given by Eqs.  \eqref{p}, \eqref{c} and $q=e^{\frac{2\pi\text{i}}{m+n\mathrm{g}}}$.

Let us now define the \emph{spectral variety} as the zero locus $V(\mathcal{I}^{(n,m)})$ of the fusion ideal $\mathcal{I}^{(n,m)}$ \eqref{Inm}:
\begin{equation}\label{sv-def}
V(\mathcal{I}^{(n,m)})=\{  \mathbf{e}=(\mathrm{e}_1,\ldots,\mathrm{e}_{n-1},1)\in\mathbb{C}^n \mid P(\mathbf{e})=0,\, \forall P\in\mathcal{I}^{(n,m)} \} .
\end{equation}

\begin{proposition}[Spectral Variety]\label{sv:prp}
For $\alpha=\frac{2\pi}{m+n\mathrm{g}}$ with  $\mathrm{g}\in (0,\infty)$, the spectral variety $V(\mathcal{I}^{(n,m)})$  \eqref{sv-def} is given by the joint spectrum
$\mathbf{e}_\nu  \bigl( {\textstyle \frac{2\pi}{m+n\mathrm{g}}},\mathrm{g};p\bigr) $ \eqref{evR:a}--\eqref{evR:c} of the truncated elliptic Ruijsenaars operators $D_r$ \eqref{Drf}:
\begin{equation}\label{sv-evR}
V(\mathcal{I}^{(n,m)})=\mathbb{E}^{(n,m)}_0= \{ \mathbf{e}_\nu  \bigl( {\textstyle \frac{2\pi}{m+n\mathrm{g}}},\mathrm{g};p\bigr)  \mid \nu \in\Lambda^{(n,m)}_0 \} .
\end{equation}
\end{proposition}

\begin{proof}
Upon specializing the parameters in the Pieri rule  \eqref{esl-pieri:r} with the aid of Proposition \ref{P-reg:prp}, one sees that for
$\lambda\in\Lambda^{(n,m)}_0$, $\mathrm{g}\in (0,\infty)$ and $\mathbf{e}\in V(\mathcal{I}^{(n,m)})$ \eqref{sv-def}:
\begin{align}\label{fPieri}
& \mathrm{e}_r P_\lambda  (\mathbf{e};{\textstyle \frac{2\pi}{m+n\mathrm{g}}},\mathrm{g};p)  = \\
&\sum_{\substack{\lambda\subset\nu\subset\lambda+1^{n} \\ |\nu|=|\lambda |+r ,\,  d_\nu\leq m}}   \psi^\prime_{\nu/\lambda}  ({\textstyle \frac{2\pi}{m+n\mathrm{g}}},\mathrm{g};p)
P_{\underline{\nu}} (\mathbf{e};{\textstyle \frac{2\pi}{m+n\mathrm{g}}},\mathrm{g};p) \quad\text{for}\ 1\leq r<n. \nonumber
\end{align}
As in the proof of Theorem \ref{e-polynomials:thm}, we can now employ Eq. \eqref{gauge} and Lemma \ref{c-reg:lem} to rewrite Eq. \eqref{fPieri} in the form
$D_r  p(\mathbf{e};{\textstyle \frac{2\pi}{m+n\mathrm{g}}},\mathrm{g};p) = \mathrm{e}_{r}   p (\mathbf{e};{\textstyle \frac{2\pi}{m+n\mathrm{g}}},\mathrm{g};p) $ with
$D_r$ and  $p_\mu (\mathbf{e};{\textstyle \frac{2\pi}{m+n\mathrm{g}}},\mathrm{g};p) $ taken from Eq. \eqref{Drf} and  Eqs. \eqref{p}, \eqref{c}, respectively.
Since $ p_0(\mathbf{e};{\textstyle \frac{2\pi}{m+n\mathrm{g}}},\mathrm{g};p) =1\neq 0$, this implies that $\mathbf{e}$ must belong to the joint spectrum $\mathbb{E}^{(n,m)}_0$ \eqref{sv-evR} of the operators $D_1,\ldots ,D_{n-1}$ in $\ell^2 (\Lambda^{(n,m)}_0,\Delta)$. Reversely, if  we assume that $\mathbf{e}\in \mathbb{E}^{(n,m)}_0$ \eqref{sv-evR} then
the eigenvalue equation entails that Eq. \eqref{fPieri} holds for $1\leq r<n$. The recurrence relations \eqref{rec:a}--\eqref{psi} thus yield in this situation that
$P_\mu(\mathbf{e};{\textstyle \frac{2\pi}{m+n\mathrm{g}}},\mathrm{g};p)=0$ if $d_\mu =m+1$   (where we use that $\psi^\prime_{\lambda+1^r/\lambda}  ({\textstyle \frac{2\pi}{m+n\mathrm{g}}},\mathrm{g};p)=1$ and that  the solution of the recurrence is unique). In other words,
the vector of joint eigenvalues $\mathbf{e}$ necessarily  lies on the zero locus $V(\mathcal{I}^{(n,m)})$ \eqref{sv-def} of $\mathcal{I}^{(n,m)}$ \eqref{Inm}.
\end{proof}

\subsection{Verlinde algebra}
Let us now define the \emph{Verlinde algebra} associated with the elliptic Ruijsenaars system as
the $\binom{n-1+m}{m}$-dimensional algebra of complex functions on the joint spectrum $\mathbb{E}^{(n,m)}_0$ \eqref{sv-evR}:
\begin{equation}\label{V-alg}
\mathcal{F}_0^{(n,m)}=\{ f: \mathbb{E}^{(n,m)}_0\to \mathbb{C} \} ,
\end{equation}
and let us write $\mathcal{R}_{0,\mathbb{C}}^{(n)}$ and $\mathcal{R}_{0,\mathbb{C}}^{(n,m)}$ for the algebras obtained by complexifying
$\mathcal{R}_{0}^{(n)}$ \eqref{R0} and $\mathcal{R}_{0}^{(n,m)}$ \eqref{f-ring}. 

\begin{proposition}[Basis for $\mathcal{F}_0^{(n,m)}$]\label{V-basis:prp}
For $\alpha=\frac{2\pi}{m+n\mathrm{g}}$ with  $\mathrm{g}>0$, the restrictions of the polynomials $P_\mu  (\mathbf{e};{\textstyle \frac{2\pi}{m+n\mathrm{g}}},\mathrm{g};p)$, $\mu\in\Lambda^{(n,m)}_0$ on the joint spectrum $\mathbb{E}^{(n,m)}_0$ \eqref{sv-evR} 
constitute a basis for the Verlinde algebra $\mathcal{F}_0^{(n,m)}$ \eqref{V-alg}.
\end{proposition}

\begin{proof}
It is immediate from the fact that the eigenfunctions $p(\mathbf{e}_\nu)$, $\nu\in\Lambda^{(n,m)}_0$ of the normal operators $D_1,\ldots, D_{n-1}$ \eqref{Drf} provide an orthogonal basis
for the Hilbert space $\ell^2(\Lambda^{(n,m)}_0,\Delta)$ (cf. \cite[Theorem 8]{die-gor:elliptic}) that
the square matrix $[ P_\mu (\mathbf{e}_\nu)]_{\mu,\nu\in\Lambda^{(n,m)}_0}$ is of full rank. 
\end{proof}

It is plain from Propositions \ref{sv:prp} and \ref{V-basis:prp} that for $\mathrm{g}\in (0,\infty)$
the vanishing ideal $\mathcal{I}(\mathbb{E}^{(n,m)}_0)=\{ P\in \mathcal{R}^{(n)}_0\mid P(\mathbf{e})=0,\, \forall \mathbf{e}\in \mathbb{E}^{(n,m)}_0\}$ is equal
to $\mathcal{I}^{(n,m)}$ \eqref{Inm}.  It means that  in this situation
the Verlinde algebra $\mathcal{F}_0^{(n,m)}$ and the
$\widehat{\mathfrak{su}}(n)_m$ elliptic fusion algebra  $\mathcal{R}_{0,\mathbb{C}}^{(n,m)}$ are isomorphic, as the kernel of
the evaluation homomorphism $P(\mathbf{e})\to P(\mathbf{e}_\nu)$ from $\mathcal{R}_{0,\mathbb{C}}^{(n)}$ onto $\mathcal{F}_0^{(n,m)}$
coincides with the complexification of $\mathcal{I}^{(n,m)}$ \eqref{Inm}.

\begin{corollary}[$\mathcal{R}_{0,\mathbb{C}}^{(n,m)}\cong\mathcal{F}_0^{(n,m)}$]\label{iso:cor}
For $\alpha=\frac{2\pi}{m+n\mathrm{g}}$ with  $\mathrm{g}\in (0,\infty)$,
the evaluation homomorphism $P(\mathbf{e})\to P(\mathbf{e}_\nu)$ from $\mathcal{R}_{0,\mathbb{C}}^{(n)}$ onto $\mathcal{F}_0^{(n,m)}$ induces an algebra isomorphism
$[P]\to P(\mathbf{e}_\nu)$ from $\mathcal{R}_{0,\mathbb{C}}^{(n,m)}$ onto $\mathcal{F}_0^{(n,m)}$.
\end{corollary}

Corollary \ref{iso:cor} makes it trivial to determine further algebraic properties of $\mathcal{R}_0^{(n,m)}$ and $\mathcal{I}^{(n,m)}$. For instance, it is manifest from the isomorphism that the factor ring $\mathcal{R}_0^{(n,m)}$ does not have nilpotents, i.e. it  is a reduced ring
and $\mathcal{I}^{(n,m)}$ is a radical ideal. However, clearly the factor ring $\mathcal{R}_0^{(n,m)}$
has zero divisors (stemming from functions in $\mathcal{F}^{(n,m)}_0$ with disjoint support), i.e.
it is not an integral domain and the ideal $\mathcal{I}^{(n,m)}$ is therefore neither maximal nor prime.

With the aid of the Verlinde algebra $\mathcal{F}_0^{(n,m)}$ we are now in the position to extend the elliptic fusion rules in Theorem \ref{e-ring:thm} so as to include the case of positive rational values
for $\mathrm{g}$.

\begin{theorem}[Structure Constants of $\mathcal{F}_0^{(n,m)}$]\label{V-algebra:thm}
For $\alpha=\frac{2\pi}{m+n\mathrm{g}}$ with $\mathrm{g}>0$, 
the structure constants of the Verlinde algebra $\mathcal{F}^{(n,m)}_0$ in the basis $P_\mu  (\mathbf{e};{\textstyle \frac{2\pi}{m+n\mathrm{g}}},\mathrm{g};p)$, $\mu\in\Lambda^{(n,m)}_0$ can be expressed in terms of elliptic Littlewood-Richardson coefficients as follows:
\begin{subequations}
\begin{align}\label{V-LR}
P_\lambda  (\mathbf{e};{\textstyle \frac{2\pi}{m+n\mathrm{g}}},\mathrm{g};p)  & P_\mu  (\mathbf{e};{\textstyle \frac{2\pi}{m+n\mathrm{g}}},\mathrm{g};p)
= \\
& \sum_{\substack{  \nu\supset \lambda,\ \nu\supset\mu \\ |\nu |= |\lambda | +|\mu | ,\, d_\nu \leq m    }} 
\left( \lim_{\substack{\mathrm{c}\in\mathbb{R}\setminus\mathbb{Q}\\ \mathrm{c}\to \mathrm{g}}} c^\nu_{\lambda,\mu} \bigl({\textstyle \frac{2\pi}{m+n\mathrm{c}} }, \mathrm{c};p\bigr) \right) P_{\underline{\nu} }   (\mathbf{e};{\textstyle \frac{2\pi}{m+n\mathrm{g}}},\mathrm{g};p) \nonumber
\end{align}
(with $\lambda ,\mu\in \Lambda^{(n,m)}_0$, $ \nu\in\Lambda^{(n)}$ and $\mathbf{e}\in\mathbb{E}^{(n,m)}_0$).

In particular, for  $1\leq r< n$, $\lambda\in\Lambda^{(n,m)}_0$ and  $\mathbf{e}\in\mathbb{E}^{(n,m)}_0$ one has explicitly that
\begin{align}\label{V-pieri:r}
P_\lambda  (\mathbf{e};{\textstyle \frac{2\pi}{m+n\mathrm{g}}},\mathrm{g};p) & P_{1^r}   (\mathbf{e};{\textstyle \frac{2\pi}{m+n\mathrm{g}}},\mathrm{g};p) = \\
&\sum_{\substack{\lambda\subset\nu\subset\lambda+1^{n} \\ |\nu|=|\lambda |+r ,\, d_\nu \leq m}}  
\psi^\prime_{\nu/\lambda}  \bigl({\textstyle \frac{2\pi}{m+n\mathrm{g}} }, \mathrm{g};p\bigr)
P_{\underline{\nu}}  (\mathbf{e};{\textstyle \frac{2\pi}{m+n\mathrm{g}}},\mathrm{g};p) . \nonumber
\end{align}
\end{subequations}
\end{theorem}

\begin{proof}
If we pick $\mathrm{g}>0$ irrational, then it is clear from  Theorem \ref{e-ring:thm} and Corollary \ref{iso:cor} that on $\mathbb{E}^{(n,m)}_0$:
\begin{align*}
P_\lambda  (\mathbf{e};{\textstyle \frac{2\pi}{m+n\mathrm{g}}},\mathrm{g};p)  & P_\mu  (\mathbf{e};{\textstyle \frac{2\pi}{m+n\mathrm{g}}},\mathrm{g};p)
= \\
& \sum_{\substack{  \nu\supset \lambda,\ \nu\supset\mu \\ |\nu |= |\lambda | +|\mu | ,\, d_\nu \leq m    }} 
 c^\nu_{\lambda,\mu} \bigl({\textstyle \frac{2\pi}{m+n\mathrm{g}} }, \mathrm{g};p\bigr)
 P_{\underline{\nu} }   (\mathbf{e};{\textstyle \frac{2\pi}{m+n\mathrm{g}}},\mathrm{g};p) ,
 \end{align*}
which establishes in particular the validity of the Pieri rule \eqref{V-pieri:r} in this situation.
Rational values of $\mathrm{g}>0$ can now be included by analytic continuation, which gives rise to Eq. \eqref{V-LR}.
Indeed, in the case of the Pieri rule the analytic continuation is achieved through the explicit formulas with the aid of
Lemma  \ref{regularity:lem} and Proposition \ref{P-reg:prp} (where one also uses that the spectral points $\mathbf{e}_\nu=\mathbf{e}_\nu(\mathrm{g},p)$ are analytic in $\mathrm{g}>0$ by
the normality of $D_r$ \eqref{Drf} in $\ell^2 (\Lambda^{(n,m)}_0, \Delta )$, cf. e.g. \cite[Chapter 2, Theorem 1.10]{kat:perturbation}).
Since  the monomials $P_{1^r}   (\mathbf{e};{\textstyle \frac{2\pi}{m+n\mathrm{g}}},\mathrm{g};p) =\mathrm{e}_r$, $1\leq r<n$ generate $\mathcal{F}_0^{(n,m)}$, the analyticity of the Pieri rule is inherited in turn  by the structure constants of
$\mathcal{F}_0^{(n,m)}$ in general. To see this, it suffices to expand one of the two factors in the product on the LHS of Eq. \eqref{V-LR} in monomials (cf. Proposition \ref{triangularity:prp}). Since the corresponding expansion coefficients are analytic by  virtue of Proposition \ref{P-reg:prp}, iterated application of the Pieri rules then leads to the  analyticity of the structure constants as claimed. 
\end{proof}

\subsection{Verlinde formula}
Let us assume that $\mathrm{g}>0$ and
denote the squared norms of the eigenbasis $p(\mathbf{e}_\nu)$, $\nu\in\Lambda^{(n,m)}_0$
 for the elliptic Ruijsenaars operators $D_r$ \eqref{Drf} in the Hilbert space by $\ell^2(\Lambda^{(n,m)}_0,\Delta)$ by (cf. Eqs. \eqref{ip}, \eqref{delta}):
\begin{equation}\label{pm}
 \hat{\Delta}_\nu= 1/\langle  p(\mathbf{e}_\nu),p(\mathbf{e}_\nu)  \rangle_\Delta \qquad (\nu\in\Lambda_0^{(n,m)}).
\end{equation}
We now endow the Verlinde algebra $\mathcal{F}^{(n,m)}_0$ \eqref{V-alg} with the following inner product:
\begin{equation}\label{dip}
\langle f ,g\rangle_{\hat{\Delta}}= \sum_{\nu\in\Lambda^{(n,m)}_0} f(\mathbf{e}_\nu)  \overline{   g(\mathbf{e}_\nu)  } \hat{\Delta}_\nu\qquad (f,g\in\mathcal{F}^{(n,m)}_0).
\end{equation}
The basis of the Verlinde algebra in Proposition \ref{V-basis:prp} is then orthogonal with respect to this inner product \cite[Corollary 11]{die-gor:elliptic}, viz. one has that
\begin{equation}\label{d-orthogonality}
\forall\lambda,\mu\in\Lambda^{(n,m)}_0:\qquad
\langle  P_\lambda , P_\mu \rangle_{\hat{\Delta}}= 
\begin{cases}
\frac{1}{c_\lambda^2 \, \Delta_\lambda} &\text{if}\ \lambda=\mu ,\\
0& \text{if}\ \lambda\neq \mu ,
\end{cases}
\end{equation}
with $c_\lambda$ and $\Delta_\lambda$ given by Eqs.  \eqref{c} and \eqref{delta}. With the aid of this  dual orthogonality  relation one arrives at a Verlinde  formula expressing
the structure constants
\begin{align}\label{VA-structure}
 P_\lambda\bigl(\mathbf{e};{\textstyle \frac{2\pi}{m+n\mathrm{g}}},\mathrm{g};p\bigr) & P_\mu\bigl(\mathbf{e};{\textstyle \frac{2\pi}{m+n\mathrm{g}}},\mathrm{g};p\bigr)
 = \\
 &
 \sum_{\kappa\in\Lambda^{(n,m)}_0}   \textsc{n}^{\kappa}_{\lambda ,\mu} \bigl( {\textstyle \frac{2\pi}{m+n\mathrm{g}}},\mathrm{g};p\bigr)
     P_\kappa\bigl(\mathbf{e};{\textstyle \frac{2\pi}{m+n\mathrm{g}}},\mathrm{g};p\bigr)\qquad (\mathbf{e}\in\mathbb{E}^{(n,m)}_0) \nonumber
\end{align}
for $\mathcal{F}^{(n,m)}_0$ in terms of a corresponding  $S$-matrix stemming from the elliptic Ruijsenaars model.

\begin{theorem}[Verlinde Formula]\label{verlinde:thm}
For $\mathrm{g}>0$ the structure constants $\textsc{n}^{\kappa}_{\lambda ,\mu} \bigl( {\textstyle \frac{2\pi}{m+n\mathrm{g}}},\mathrm{g};p\bigr)$ \eqref{VA-structure}  for  the Verlinde algebra $\mathcal{F}^{(n,m)}_0$ \eqref{V-alg} are given by
\begin{subequations}
\begin{equation}\label{verlinde}
 \textsc{n}^{\kappa}_{\lambda ,\mu} \bigl( {\textstyle \frac{2\pi}{m+n\mathrm{g}}},\mathrm{g};p\bigr)
 =   \sum_{\nu\in\Lambda^{(n,m)}_0}   \frac{S_{\lambda,\nu} S_{\mu,\nu}  S^{-1}_{\nu,\kappa}}{S_{0,\nu}}   
\end{equation}
with
\begin{equation}\label{S-matrix}
S_{\lambda,\nu} =  S_{\lambda,\nu}  \bigl( {\textstyle \frac{2\pi}{m+n\mathrm{g}}},\mathrm{g};p\bigr)  = 
\frac{P_\lambda\bigl(\mathbf{e}_\nu ;{\textstyle \frac{2\pi}{m+n\mathrm{g}}},\mathrm{g};p\bigr)}{ c_{\nu} \bigl( {\textstyle \frac{2\pi}{m+n\mathrm{g}}},\mathrm{g};p\bigr) }
\end{equation}
and
\begin{equation}\label{S-inverse}
S^{-1}_{\lambda,\nu} = 
c_\lambda^2 \hat{\Delta}_\lambda \overline{S_{\nu,\lambda}  }  c_\nu^2  \Delta_\nu .
\end{equation}
\end{subequations}
\end{theorem}

\begin{proof}
Let us first
observe that it is clear from the orthogonality relation in Eq. \eqref{d-orthogonality} that the inverse of the $S$-matrix $S_{\lambda,\nu}$ \eqref{S-matrix} is given by
the matrix $S^{-1}_{\lambda,\nu} $ \eqref{S-inverse}. Hence, upon pairing both sides of Eq. \eqref{VA-structure} against $P_\kappa\bigl(\mathbf{e};{\textstyle \frac{2\pi}{m+n\mathrm{g}}},\mathrm{g};p\bigr)$, $\kappa\in\Lambda^{(n,m)}_0$
with $\langle \cdot ,\cdot\rangle_{\hat{\Delta}}$ \eqref{dip}, the Verlinde formula readily follows via the orthogonality relation \eqref{d-orthogonality}:
\begin{align*}
\textsc{n}^{\kappa}_{\lambda ,\mu} \bigl( {\textstyle \frac{2\pi}{m+n\mathrm{g}}},\mathrm{g};p\bigr)
 &= \frac{\langle   P_\lambda  P_\mu , P_\kappa  \rangle_{\hat{\Delta}} }{\langle  P_\kappa , P_\kappa \rangle_{\hat{\Delta}} } 
 =c_\kappa^2  \Delta_\kappa \sum_{\nu\in\Lambda^{(n,m)}_0 }  P_\lambda  (\mathbf{e}_\nu )  P_\mu   (\mathbf{e}_\nu )  \overline{ { P_\kappa   (\mathbf{e}_\nu ) }  }   \hat{\Delta}_\nu \\
 &=  \sum_{\nu\in\Lambda^{(n,m)}_0}   \frac{S_{\lambda,\nu} S_{\mu,\nu} S^{-1}_{\nu,\kappa}}{S_{0,\nu}} .
\end{align*}
\end{proof}

It is useful to recall though that we already know from Theorem \ref{V-algebra:thm} that many of the structure constants in Eq. \eqref{VA-structure} in fact vanish:
\begin{align*}
& \textsc{n}^{\kappa}_{\lambda ,\mu} \bigl( {\textstyle \frac{2\pi}{m+n\mathrm{g}}},\mathrm{g};p\bigr)= \\
& \begin{cases}
 \lim_{\substack{\mathrm{c}\in\mathbb{R}\setminus\mathbb{Q}\\ \mathrm{c}\to \mathrm{g}}} c^{\nu}_{\lambda,\mu} \bigl({\textstyle \frac{2\pi}{m+n\mathrm{c}} }, \mathrm{c};p\bigr) 
 &\text{if}\ \kappa=\underline{\nu} \ \text{with}\ \nu\supset \lambda,\,  \nu\supset \mu,\, \text{and}\ |\nu |=|\lambda|+|\mu|,\\
\qquad 0& \text{otherwise}.
\end{cases}
\end{align*}
Notice also that it follows from Eq. \eqref{S-inverse} that the absolute value of the determinant of the $S$-matrix is given by
\begin{equation}
 |\det S |= \Bigl( \prod_{\lambda\in\Lambda^{(n,m)}_0}   c_\lambda^2  \sqrt{\Delta_\lambda \hat{\Delta}_\lambda} \Bigr)^{-1} .
\end{equation}

\subsection{Degenerations}
To see how the structure constants of the  $\widehat{\mathfrak{su}}(n)_m$ Wess-Zumino-Witten fusion ring and its refined deformation are recovered from $\mathcal{F}^{(n,m)}_0$,
let us put
\begin{equation}
q=e^{\frac{2\pi \mathrm{i}}{m+n\mathrm{g}}} \quad\text{and}\  \mathrm{e}_r=(x_1\cdots x_n)^{-r/n} m_{1^r}(x) \quad 1\leq r<n.
\end{equation}
Then for $\mathrm{g}>0$, $-1<p<1$ and $\lambda,\mu \in\Lambda^{(n,m)}_0$ we have that:
\begin{subequations}
\begin{align}\label{Pieri-deg}
&\psi^\prime_{\nu/\lambda}  \bigl({\textstyle \frac{2\pi}{m+n\mathrm{g}} }, \mathrm{g};p\bigr)=\\
&\begin{cases}
 \prod_{\substack{1\leq j<k\leq n\\ \theta_j-\theta_k=-1}} 
 {\textstyle 
  \frac{[\nu_j-\nu_k+\mathrm{g}(k-j+1)]_q}{[\nu_j-\nu_k+\mathrm{g}(k-j)]_q}
 \frac{[\lambda_j-\lambda_k+\mathrm{g}(k-j-1)]_q}{[\lambda_j-\lambda_k+\mathrm{g}(k-j)]_q}   }  &\text{if}\ p=0,\\
1 &\text{if}\ \mathrm{g}=1
\end{cases} \nonumber
\end{align}
(for $\lambda \subset\nu\subset\lambda+1^n$ with $\theta=\nu-\lambda$ and $\underline{\nu}\in\Lambda^{(n,m)}_0$),
\begin{equation}\label{P-deg}
P_\mu\bigl( \mathbf{e} ;{\textstyle \frac{2\pi}{m+n\mathrm{g}}},\mathrm{g};p\bigr)= 
\begin{cases}
(x_1\cdots x_n)^{-|\mu|/n}  P_\mu (x;q,q^g)&\text{if}\ p=0,\\
(x_1\cdots x_n)^{-|\mu |/n}  s_\mu (x)&\text{if}\ \mathrm{g}=1,
\end{cases} 
\end{equation}
and
\begin{align}\label{spectrum-deg}
& \mathrm{e}_{r,\nu}  \bigl( {\textstyle \frac{2\pi}{m+n\mathrm{g}}},\mathrm{g};p\bigr)  = \\
&\begin{cases}
 q^{-r \bigl( \frac{|\nu|}{n}+\frac{(n-1)\mathrm{g}}{2}\bigr)}
m_{1^r} (q^{\nu_1+(n-1)\mathrm{g}},q^{\nu_2+(n-2)\mathrm{g}},\ldots, q^{\nu_{n-1}+\mathrm{g}},1)  &\text{if}\ p=0,\\
q^{-r \bigl( \frac{|\nu|}{n}+\frac{n-1}{2}\bigr)}
m_{1^r} (q^{\nu_1+n-1},q^{\nu_2+n-2},\ldots, q^{\nu_{n-1}+1},1)  &\text{if}\ \mathrm{g}=1
\end{cases}\nonumber
\end{align}
\end{subequations}
(for $\nu\in\Lambda^{(n,m)}_0$).
Indeed, Eqs. \eqref{Pieri-deg} and \eqref{P-deg} follow from Proposition \ref{LRdegenerations:prp} and  Corollary \ref{Pdegenerations:cor} by Lemma \ref{regularity:lem}, while
Eq. \eqref{spectrum-deg} is immediate from Eq. \eqref{evR:c} when $p=0$, and hence also when $\mathrm{g}=1$ (since the $p$-dependence drops out at  $\mathrm{g}=1$).

At $p=0$, Theorem \ref{V-algebra:thm} therefore computes the structure constants of the refined Verlinde algebra  \cite{aga-sha:knot,che:double,kir:inner,nak:refined} in terms of Macdonald's $(q,t)$-Littlewood-Richardson coefficients
(cf. Proposition \ref{LRdegenerations:prp}):
\begin{subequations}
\begin{align}\label{refined-fusion}
& \textsc{n}^{\kappa}_{\lambda ,\mu} \bigl( {\textstyle \frac{2\pi}{m+n\mathrm{g}}},\mathrm{g};0\bigr)= \\
& \begin{cases}
 \lim_{\substack{\mathrm{c}\in\mathbb{R}\setminus\mathbb{Q}\\ \mathrm{c}\to \mathrm{g}}} 
 f^{\nu}_{\lambda,\mu} \bigl( e^{\textstyle \frac{2\pi \mathrm{i}}{m+n\mathrm{c}} }, e^{\textstyle \frac{2\pi  \mathrm{i} \mathrm{c}}{m+n\mathrm{c}} } \bigr) 
 &\text{if}\ \kappa=\underline{\nu} \ \text{with}\ \nu\supset \lambda,\,  \nu\supset \mu,\, \text{and}\ |\nu|=|\lambda|+|\mu|,\\
\qquad 0& \text{otherwise} .
\end{cases} \nonumber
\end{align}
In the case of the Pieri rule this becomes explicitly
\begin{align}\label{refined-pieri}
& \textsc{n}^{\kappa}_{\lambda ,1^r} \bigl( {\textstyle \frac{2\pi}{m+n\mathrm{g}}},\mathrm{g};0\bigr)= \\
& \begin{cases}
\psi^\prime_{\nu/\lambda}  \bigl({\textstyle \frac{2\pi}{m+n\mathrm{g}} }, \mathrm{g};0\bigr) &\text{if}\ \kappa=\underline{\nu} \ \text{with}\  \lambda\subset \nu \subset \lambda+1^n,\, \text{and} \ |\nu|=|\lambda|+r ,\\
\qquad 0& \text{otherwise} .
\end{cases} \nonumber
\end{align}
\end{subequations}

At $\mathrm{g}=1$, on the other hand, we see from the Pieri rule in Theorem \ref{V-algebra:thm} that
\begin{subequations}
\begin{align}\label{WZW-pieri}
 \textsc{n}^{\kappa}_{\lambda ,1^r} \bigl( {\textstyle \frac{2\pi}{m+n}},1;p\bigr) &=
 \textsc{n}^{\kappa}_{\lambda ,1^r} \bigl( {\textstyle \frac{2\pi}{m+n}},1;0\bigr) \\
& \begin{cases}
1 &\text{if}\ \kappa=\underline{\nu} \ \text{with}\  \lambda\subset \nu\subset \lambda+1^n,\, \text{and} \ |\nu |=|\lambda|+r ,\\
0& \text{otherwise} .
\end{cases} \nonumber
\end{align}
More generally, Theorem \ref{V-algebra:thm} thus states that
\begin{align}\label{fusion}
& \textsc{n}^{\kappa}_{\lambda ,\mu} \bigl( {\textstyle \frac{2\pi}{m+n}},\mathrm{1};p\bigr)= \textsc{n}^{\kappa}_{\lambda ,\mu} \bigl( {\textstyle \frac{2\pi}{m+n}},\mathrm{1};0\bigr) =\\
& \begin{cases}
 \lim_{\substack{\mathrm{c}\in\mathbb{R}\setminus\mathbb{Q}\\ \mathrm{c}\to \mathrm{1}}} 
 f^{\nu}_{\lambda,\mu} \bigl( e^{\textstyle \frac{2\pi \mathrm{i}}{m+n\mathrm{c}} }, e^{\textstyle \frac{2\pi  \mathrm{i} \mathrm{c}}{m+n\mathrm{c}} } \bigr) 
 &\text{if}\ \kappa=\underline{\nu} \ \text{with}\ \nu\supset \lambda,\,  \nu\supset \mu,\, \text{and}\ |\nu |=|\lambda|+|\mu|,\\
\qquad 0& \text{otherwise},
\end{cases}  \nonumber
\end{align}
\end{subequations}
which retrieves the structure constants of the 
$\widehat{\mathfrak{su}}(n)_m$ Wess-Zumino-Witten fusion ring \cite{dif-mat-sen:conformal,fuc:affine} through Macdonald's $(q,t)$-Littlewood-Richardson coefficients.
In other words, \emph{$(q,t)$-deformation can be used as a vehicle for computing structure constants in the fusion ring (i.e. modulo the fusion ideal) by degeneration from
(deformed) Littlewood-Richardson coefficients in the ring of symmetric polynomials itself}, as was previously pointed out in  \cite{die:sunm}.

Finally, we see from Eqs. \eqref{P-deg}, \eqref{spectrum-deg} that the corresponding degenerations of the elliptic $S$-matrix \eqref{S-matrix} are given by
\begin{subequations}
\begin{align}\label{RWZW-S}
  S_{\lambda,\nu}  \bigl( {\textstyle \frac{2\pi}{m+n\mathrm{g}}},\mathrm{g};0\bigr)  = &q^{-\frac{1}{n}|\lambda | |\nu |-\frac{1}{2}(n-1)( | \lambda |+|\nu |)\mathrm{g} } \\
& \times  P_\lambda ( q^{\nu_1+(n-1)\mathrm{g}}, q^{\nu_2+(n-2)\mathrm{g}},\ldots ,q^{\nu_{n-1}+\mathrm{g}},1; q,q^{\mathrm{g}}) 
\nonumber \\ &\times  
P_\nu ( q^{(n-1)\mathrm{g}}, q^{(n-2)\mathrm{g}},\ldots ,q^{\mathrm{g}},1 ; q,q^{\mathrm{g}}) , \nonumber
 \end{align}
 and 
 \begin{equation}\label{WZW-S}
  S_{\lambda,\nu}  \bigl( {\textstyle \frac{2\pi}{m+n}},1;p\bigr) = S_{\lambda,\nu}  \bigl( {\textstyle \frac{2\pi}{m+n}},1;0\bigr) 
 \frac{ c_{\nu} \bigl( {\textstyle \frac{2\pi}{m+n}},1;0\bigr)  }{ c_{\nu} \bigl( {\textstyle \frac{2\pi}{m+n}},1;p\bigr) } 
 \end{equation}
 \end{subequations}
  with
 \begin{align*}
 S_{\lambda,\nu}  \bigl( {\textstyle \frac{2\pi}{m+n}},1;0\bigr) =&  q^{-\frac{1}{n}|\lambda | |\nu |-\frac{1}{2}(n-1)( | \lambda |+|\nu |)} \\
& \times  s_\lambda ( q^{\nu_1+n-1}, q^{\nu_2+n-2},\ldots ,q^{\nu_{n-1}+1},1) \nonumber \\
&\times  s_\nu ( q^{n-1}, q^{n-2},\ldots ,q,1) . \nonumber
 \end{align*}
 Notice that the gauge factor multiplying the $S$-matrix in Eq. \eqref{WZW-S} from the right cancels in the Verlinde formula \eqref{verlinde}, and  observe also that
 we have  rewritten   the final expression with the aid of the principal specialization formulas for the Macdonald and Schur polynomials  \cite[Chapter VI,\, \S 6]{mac:symmetric}:
 \begin{equation*}
 q^{-\frac{1}{2}  |\nu | (n-1)\mathrm{g}}P_\nu ( q^{(n-1)\mathrm{g}}, q^{(n-2)\mathrm{g}},\ldots ,q^{\mathrm{g}},1; q,q^{\mathrm{g}})  =
 \prod_{1\leq j<k\leq n} {\textstyle \frac{[ (k-j+1)\mathrm{g}]_{q,\nu_j-\nu_k}}{[ (k-j)\mathrm{g}]_{q, \nu_j-\nu_k}} } 
 \end{equation*}
($=1/ c_\nu(\frac{2\pi}{m+n\mathrm{g}},\mathrm{g}; 0)$) and
 \begin{equation*}
 q^{-\frac{1}{2}  |\nu | (n-1)} s_\nu ( q^{n-1}, q^{n-2},\ldots ,q,1)  =
 \prod_{1\leq j<k\leq n} {\textstyle \frac{[ k-j+\nu_j-\nu_k]_{q}}{[ k-j ]_{q} } }
 \end{equation*}
($=1/ c_\nu(\frac{2\pi}{m+n},1; 0)$),
where $[z]_{q,k}=\prod_{0\leq l<k} [z+l]_q$ with $ [z]_{q, 0}=1$.

We thus observe that $S_{\lambda,\nu}  \bigl( {\textstyle \frac{2\pi}{m+n}},1;0\bigr) $
coincides up to normalization  with (the adjoint of) the Kac-Peterson modular $S$-matrix for $\widehat{\mathfrak{su}}(n)_m$  (cf. e.g.
 \cite[Theorem 13.8]{kac:infinite}, \cite[Equation (14.217)]{dif-mat-sen:conformal} and \cite[Proposition 6.15]{kor-str:slnk}), while
 the $S$-matrix $S_{\lambda,\nu}  \bigl( {\textstyle \frac{2\pi}{m+n\mathrm{g}}},\mathrm{g};0\bigr) $
recovers in turn the trigonometric $S$-matrix from the refined Chern-Simons theory of knot invariants (cf. e.g.
\cite[Section 5.2]{aga-sha:knot}, \cite[Section 2.6]{gor-neg:refined},  \cite[Section 5]{kir:inner} and \cite[Section 3]{nak:refined}).
In other words, Theorem \ref{verlinde:thm} reproduces in these two situations, respectively, the classical Verlinde formula for the structure constants of the 
$\widehat{\mathfrak{su}}(n)_m$ Wess-Zumino-Witten fusion ring (cf.  e.g. \cite[Exercise 13.34]{kac:infinite}, \cite[Equation (16.3)]{dif-mat-sen:conformal}, \cite[Equation (6.40]{kor-str:slnk}) as well as its  refinement  stemming from Macdonald theory \cite{aga-sha:knot,die:sunm}. The conventional normalization of the $S$-matrix corresponds to
the following rescaling:
\begin{subequations}
\begin{equation}
S_{\lambda,\nu}  \bigl( {\textstyle \frac{2\pi}{m+n\mathrm{g}}},\mathrm{g};p\bigr)\to 
\mathrm{n}^{-\frac{1}{2}}  \bigl( {\textstyle \frac{2\pi}{m+n\mathrm{g}}},\mathrm{g};p\bigr)
S_{\lambda,\nu}  \bigl( {\textstyle \frac{2\pi}{m+n\mathrm{g}}},\mathrm{g};p\bigr)
\end{equation}
with
\begin{equation}
\mathrm{n} \bigl( {\textstyle \frac{2\pi}{m+n\mathrm{g}}},\mathrm{g};p\bigr)= \sum_{\lambda\in\Lambda^{(n,m)}_0}  \Delta_\lambda \bigl({\textstyle \frac{2\pi}{m+n\mathrm{g}}},\mathrm{g};p\bigr) .
\end{equation}
\end{subequations}
Notice in this connection that at the classical parameter specialization $(\mathrm{g},p)=(1,0)$ this normalization recovers (the adjoint of) the unitary Kac-Peterson
$S$-matrix for $\widehat{\mathfrak{su}}(n)_m$:
\begin{equation*}
\mathrm{n} \bigl( {\textstyle \frac{2\pi}{m+n}},1;0\bigr)=
 \sum_{\lambda\in\Lambda^{(n,m)}_0} 
\prod_{1\leq j<k\leq n}
 {\textstyle
\frac{[\lambda_j-\lambda_k+k-j]_q^2}{[k-j]_q^2} } =  \sum_{\lambda\in\Lambda^{(n,m)}_0} \prod_{1\leq j<k\leq n}
 S_{0,\lambda }^2 \bigl( {\textstyle \frac{2\pi}{m+n}},1;0\bigr)  
 \end{equation*}
(so
$\mathrm{n} \bigl( {\textstyle \frac{2\pi}{m+n}},1;0\bigr) =  \frac{(2\sin \frac{\pi}{m+n})^{-n(n-1)} n(n+m)^{n-1}}{ \prod_{1\leq j<k\leq n} [k-j]_q^2} $, cf. e.g. \cite[Proposition 6.15]{kor-str:slnk}).

\section*{Acknowledgements}
Helpful feedback from Stephen Griffeth  is gratefully acknowledged.

The work of JFvD was supported in part by the {\em Fondo Nacional de Desarrollo
Cient\'{\i}fico y Tecnol\'ogico (FONDECYT)} Grant \# 1210015. TG was supported in part by the NKFIH Grant K134946.

\bigskip\noindent
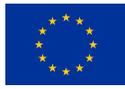
\parbox{.135\textwidth}{\begin{tikzpicture}[scale=.03]
\fill[fill={rgb,255:red,0;green,51;blue,153}] (-27,-18) rectangle (27,18);  
\pgfmathsetmacro\inr{tan(36)/cos(18)}
\foreach \i in {0,1,...,11} {
\begin{scope}[shift={(30*\i:12)}]
\fill[fill={rgb,255:red,255;green,204;blue,0}] (90:2)
\foreach \x in {0,1,...,4} { -- (90+72*\x:2) -- (126+72*\x:\inr) };
\end{scope}}
\end{tikzpicture}} \parbox{.85\textwidth}{This project has received funding from the European Union's Horizon 2020 research and innovation programme under the Marie Sk{\l}odowska-Curie grant agreement No 795471.}

\bibliographystyle{amsplain}

\end{document}